%% file: convergence.tex
\documentclass{amsart}
\usepackage[margin=1in]{geometry}

\usepackage{amssymb}
\usepackage{amsthm}
\usepackage{amsmath}
\usepackage{graphicx}
\usepackage{mathtools}
\usepackage{caption}
\usepackage{pgfplots}
\usepackage{footnote}
\usepackage{amsrefs}

\makeatletter
\newcommand*\bigcdot{\mathpalette\bigcdot@{.5}}
\newcommand*\bigcdot@[2]{\mathbin{\vcenter{\hbox{\scalebox{#2}{$\m@th#1\bullet$}}}}}
\makeatother

\newcommand{\ZZ}{\mathbb{Z}}

\newcommand{\RR}{\mathbb{R}}

\newcommand{\norm}[1]{\left\lVert#1\right\rVert}

\pgfplotsset{compat=1.10}
\usetikzlibrary{shapes, positioning}

\makeatletter

\newcommand{\bbeta}{\boldsymbol\beta}
\newcommand{\bmu}{\boldsymbol\mu}
\newcommand{\bn}{\mathbf{n}}

\makeatletter
\renewcommand*\env@matrix[1][\arraystretch]{%
  \edef\arraystretch{#1}%
  \hskip -\arraycolsep
  \let\@ifnextchar\new@ifnextchar
  \array{*\c@MaxMatrixCols c}}
\makeatother
\makesavenoteenv{tabular}

\def\doclabel#1{\gdef\@doclabel{#1}}
\doclabel{Use {\tt\textbackslash doclabel\{MY LABEL\}}.}
\def\docdate#1{\gdef\@docdate{#1}}
\docdate{Use {\tt\textbackslash docdate\{MY DATE\}}.}
\def\docauthor#1{\gdef\@docauthor{#1}}
\DeclareMathOperator{\Span}{\text{span}}

\DeclareMathOperator*{\diam}{\text{diam}}

\DeclareMathOperator{\argmin}{\textrm{argmin}}

\DeclareMathOperator{\conv}{\textrm{conv}}

\newcommand{\quaderror}{\mathcal{E}}
\newcommand{\quaderrorlocal}{\mathcal{E}_T}
\newcommand{\quaderrorref}{\hat{\mathcal{E}}}
\newcommand{\Tref}{\hat{T}}
\DeclareMathOperator{\interior}{\textrm{int}}
\newcommand{\galerkin}[1]{\Bar{#1}}
\newcommand{\galerkinsolution}{\galerkin{u}_h}

\newcommand{\galerkinform}[1]{\galerkin{\mathcal{#1}}}
\newcommand{\fembasis}{\mathcal{B}_h}
\newcommand{\bilinear}[3]{#1(#2,#3)}
\DeclareMathOperator{\diag}{\textrm{diag}}
\usepackage{hyperref}

\newcommand{\ext}[1]{{#1}_{\textrm{ext}}}

\email[Max Heldman]{maxh@vt.edu}

\title[A monotone finite element method for reaction-drift-diffusion equations]{A monotone finite element method for reaction-drift-diffusion equations with discontinuous reaction coefficients} 
\author{M. Heldman}
\thanks{M.H. was partially supported by the National Science Foundation (NSF-DMS 1902854 and a subgrant of NSF-OAC 2139536) and the Army Research Office (ARO W911NF2010244)}
\date{\today}
\begin{document}

\theoremstyle{definition}
\newtheorem{theorem}{Theorem}[section]
\newtheorem{definition}[theorem]{Definition}
\newtheorem{proposition}[theorem]{Proposition}
\newtheorem{example}[theorem]{Example}
\newtheorem{problem}[theorem]{Problem}
\newtheorem{comment}[theorem]{comment}
\newtheorem{lemma}[theorem]{Lemma}
\newtheorem{assumption}{Assumption}
\newtheorem{corollary}[theorem]{Corollary}
\newtheorem{remark}[theorem]{Remark}

\graphicspath{{img/}}
\input{sections2/introduction}

\bibliographystyle{plain}
\bibliography{convergence.bib}

\end{document}

%% file: sections2/introduction.tex
\begin{abstract}
We prove an abstract convergence result for a family of dual-mesh based quadrature rules on tensor products of simplical meshes. In the context of the multilinear tensor-product finite element discretization of reaction-drift-diffusion equations, our quadrature rule generalizes the mass-lump rule, retaining its most useful properties; for a nonnegative reaction coefficient, it gives an $O(h^2)$-accurate, nonnegative diagonalization of the reaction operator. The major advantage of our scheme in comparison with the standard  mass lumping scheme is that, under mild conditions, it produces an $O(h^2)$ consistency error even when the integrand has a jump discontinuity. The finite-volume-type quadrature rule has been stated in a less general form and applied to systems of reaction-diffusion equations related to particle-based stochastic reaction-diffusion simulations (PBSRD); in this context, the reaction operator is \textit{required} to be an $M$-matrix and a standard model for bimolecular reactions has a discontinuous reaction coefficient. We apply our convergence results to a finite element discretization of scalar drift-diffusion-reaction model problem related to PBSRD systems, and provide new numerical convergence studies confirming the theory.
\end{abstract}

\maketitle



\section{Introduction}\label{section:introduction}
In this paper, we prove and apply convergence results for the finite-volume-type quadrature rule
 \begin{equation}\label{eq:quadrature_scheme}
 \begin{aligned}
 \int_\Omega u_h(x)v_h(x)\lambda_1(x)\lambda_2(x)1_K(x)dx &\approx \sum_{i=1}^N u_iv_i \lambda_1(x_i)\int_{V_i} \lambda_2(x)1_K(x)dx \\&=: \mathcal{Q}(u_h,v_h) ,
 \end{aligned}
 \end{equation} where $K\subseteq \Omega \subseteq \RR^d$ with indicator function $1_K$, $\Omega$ is a polygonal domain, $\{x_i\}_{i=1}^N$ are the vertices of a triangulation $\mathcal{T}_h$ of $\Omega$ (the primal mesh), $\{V_i\}_{i=1}^N$ is a chosen polygonal partition of $\Omega$ (the dual mesh to $\mathcal{T}_h$), $\lambda_1,\lambda_2 : \RR^d\to \RR$ are sufficiently smooth nonnegative functions, and finally $u_h$ and $v_h$ are drawn from a finite-element space $V_h$. Here and throughout this paper, we use the notation $w_h(x_i) = w_i$ for any $w_h\in V_h$. The product $\lambda(x) = \lambda_1(x)\lambda_2(x)$ represents a chosen multiplicative splitting of $\lambda$ into components $\lambda_1,\lambda_2 \geq 0$. We also define $\kappa(x) := \lambda(x)1_K(x)$. In this setup, therefore, the splitting of $\lambda$ is chosen for the convenience of the application, whereas the set $K$ is fixed. 

 The inspiration for this research is a domain-specific family of coupled reaction-drift-diffusion equation \textit{systems} of parabolic partial integro-differential equations (PIDEs) related to particle-based stochastic reaction-diffusion (PBSRD) simulation. In that context, \eqref{eq:quadrature_scheme} was derived and applied in a less general form (i.e., with specific choices of set $K$ and smooth weighting function $\lambda(x)$) in \cites{isaacson_convergent_2013,isaacson_unstructured_2018} and more recently in \cites{isaacson_unstructured_2024,heldman_convergent_2024}, to produce bimolecular reaction rates as part of the convergent reaction-diffusion master equation (CRDME) simulation methodology. In the cited papers, convergence of \textit{statistics} for the system solution were investigated numerically but not analytically. In the present work, we apply our quadrature error estimates for \eqref{eq:quadrature_scheme} to obtain $H^1(\Omega)$ and $L^2(\Omega)$ error estimates for the finite element approximation of the solution to a simplified model problem for PBSRD systems, the scalar reaction-drift-diffusion partial differential equation (PDE). In strong form, the model problem reads as follows: find $u\in H^1(\Omega)$ such that
\begin{equation}\label{eq:reaction-diffusion-parabolic-strong}
\begin{aligned}
    -\nabla \cdot [ \alpha(x)\nabla u(x,t) + u\bbeta ] &= -\kappa(x)u(x,t) + f(x),\qquad &&x\in\Omega,\\
    u(x,t) &= g_D(x)\qquad &&x\in  \Gamma_D, \\
    J(u) \cdot \bn_\Omega(x) &= g_N(x)\qquad &&x\in  \Gamma_N,
\end{aligned}
\end{equation}
where $J(u) = \alpha(x) \nabla u(x,t) + u\bbeta$, $\Gamma_D \cup \Gamma_N = \partial\Omega$, $\bn_\Omega(x)$ is the outward pointing normal vector for $\Omega$, the reaction coefficient $\kappa: \RR^d\to\RR$ is of the same form as in the previous paragraph, and the transport operator coefficients $\alpha : \RR^d \to \RR^d$ and $\bbeta:\RR^d\to\RR^d$, forcing function $f:\RR^d\to \RR$, and boundary data $g_D:\RR^d\to \RR$ and $g_N:\RR^d\to \RR$ are assumed to be sufficiently smooth. In the parlance of the finite element method literature, computing the integrals appearing in the standard Galerkin method for \eqref{eq:reaction-diffusion-parabolic-strong} using quadrature rules such as \eqref{eq:quadrature_scheme} introduces a \textit{consistency} error; in that context, the results in Section \ref{section:functionspacesandquadratureerrors} and \ref{section:refinedconvergenceresults} are referred to as consistency error estimates which lead to the convergence proofs in Section \ref{section:reactiondiffusionequations}. 

A distinguishing property of our quadrature scheme is that neither the primal nor the dual mesh needs to have any explicit relation to $K$. Therefore, the method is targeted toward applications where it is difficult or impossible to align the mesh with the discontinuity. Again, our target application is the CRDME for which the high dimensionality of the PIDE system prevents us from fitting the mesh to the interface. Another possible application includes evolving level-set problems for which it may be more efficient to recompute, or approximately recompute, the discrete reaction operator near the evolving interface $\partial K$ than to adjust the mesh itself. The other main properties of our scheme are its monotonicity and accuracy, producing diagonal discrete reaction operators and quasi-optimal error estimates in $H^1(\Omega)$ and $L^2(\Omega)$ for the finite element approximation of \eqref{eq:reaction-diffusion-parabolic-strong} when combined with an appropriate discretization of the transport operator. It can therefore be viewed as a generalization of the mass-lump method (see, e.g., \cite{thomee_galerkin_2006}). For the CRDME, the nonnegative diagonalization property is not only desirable but \textit{necessary} to the methodology. 

The majority of the paper (Sections \ref{section:basiconvergenceresults} and \ref{section:refinedconvergenceresults}) is devoted to proving convergence rate estimates for \eqref{eq:quadrature_scheme}, particularly local and global $O(h)$ estimates in Lemma \ref{lemma:H1consistency} and local $O(h^2)$ error estimates away from the discontinuity in Lemma \ref{lemma:L2consistency}. Then, in Section \ref{section:refinedconvergenceresults}, we exploit the fact that only a small number of elements lie along the interface to produce improved convergence rates. A straightforward but coarse estimate using H\"older's inequality yields an improved bound applicable to proving supercloseness results between the quadrature solution and the Galerkin solution, while a more refined analysis using a uniform trace inequality over level sets of the signed distance function from the interface $\partial K$ finally results in the $O(h^2)$ error estimate of Corollary \ref{corollary:final_global_error_estimate}. As far as we are aware, a proof of the preliminary estimate Lemma \ref{lemma:H1consistency} has not appeared in the literature, although it uses standard techniques. However, as we discuss in more detail in Remark \ref{remark:related_work}, Lemma \ref{lemma:L2consistency} has appeared in the literature in a less general form in the context of analyzing the close relationship between piecewise linear Galerkin finite element methods and finite-volume-element methods. In particular, \cite{hackbusch_first_1989,liang_symmetric_2002,ma_symmetric_2003,xu_analysis_2009} contain similar analyses toward proving closeness and supercloseness estimates between the solutions of the finite-volume-element method and the Galerkin finite-element method. 

The estimates on the interface are more delicate and are closely related to some results from the papers \cite{zhang_analysis_2008,kashiwabara_finite_2023}, among others, which analyze finite element errors localized to a neighborhood of the boundary of the domain $\Omega$ (so-called boundary skin estimates). In addition to being a novel application of the arguments appearing in the cited sources, Lemma \ref{lemma:uniformtraceinequality} and Proposition \ref{proposition:controlonKh} in Section \ref{section:uniformtraceinequality} give a new proof, and generalization of, a useful result appearing in several of those papers (specifically \cite{kashiwabara_penalty_2016}, Theorem 8.3 and \cite{zhang_analysis_2008}, Lemma 2.1). The upshot of the analysis is that even with local quadrature errors of $O(h^{k-1})$, one can recover $O(h^k)$ convergence rates for the approximate solution as long as the low order errors are localized to a low-dimensional set. Many problems have singularities living on low-dimensional manifolds (as in the current paper), or require modified, lower-order discretizations near a subdomain boundary. Therefore, the arguments in this paper hint at a general principle that we hope to apply in future work.

The remainder of our paper is organized as follows: In Section \ref{section:basiconvergenceresults}, we provide additional details and assumptions related to the meshes (primal and dual) and function space $V_h$ mentioned in the text around \eqref{eq:quadrature_scheme} above. Then, we give a short introduction to the techniques involved in the proofs of the preliminary estimates, and subsequently prove the results. In Section \ref{section:refinedconvergenceresults}, we derive improved global error estimates using the techniques mentioned in the previous paragraph. In Section \ref{section:reactiondiffusionequations}, we briefly describe the application of our quadrature error results to prove the convergence of a finite element scheme for \eqref{eq:reaction-diffusion-parabolic-strong}, mainly by applying known results such as the Strang Lemma to turn the consistency error estimates into global error bounds. In that section we also describe how to modify the discrete drift-diffusion operator in the tensor product case to obtain a fully monotone discretization. In Section \ref{section:numerics}, we give numerical results for a model problem related to PBSRD systems. Finally, Section \ref{section:conclusion} concludes the paper.

\section{Setup and basic convergence results}\label{section:basiconvergenceresults}

Let $\Omega\subseteq \RR^d$ be a polygonal domain, and let $\mathcal{T}_h$ be a family of subdivisions of $\Omega$ indexed by $h$, where $h =\max_{T\in\mathcal{T}_h} \diam T$ is the mesh size of $\mathcal{T}_h$. We restrict $T\in \mathcal{T}_h$ to be a tensor product of simplical elements on one or more domains. Specifically:
     \begin{itemize}
        \item[(i)] $\Omega = \Omega_1 \times \Omega_2 \times ... \times \Omega_J$ for some $J > 0$, where each $\Omega_j \subseteq \RR^{d_j}$ is a polygonal domain with $d = \sum_{j=1}^J d_j$,
        \item[(ii)] $\mathcal{T}_h^{(j)}$ is a simplical subdivision of $\Omega_j$, and  $$\mathcal{T}_h = \left\{\prod_{j=1}^J T_j : T_j \in \mathcal{T}_h^{(j)}\right\}.$$
    \end{itemize}
Given our assumptions on $\mathcal{T}_h$, we observe that each element $T\in \mathcal{T}_h$ has the same number of nodes $$N_T = \prod_{j=1}^J (d_j + 1).$$

We denote the vertices of $\mathcal{T}_h$ by $\{x_i\}_{i=1}^N$, and note that $$x_i = (x^{(1)}_{i_1},x^{(2)}_{i_2},...,x^{(J)}_{i_J}),$$ where $x^{(j)}_{i_j}$, $j=1,2,...,J$, is a vertex of $\mathcal{T}_h^{(j)}$.

 For the purposes of this paper, we will refer to a function $f(x) = f(x_1,x_2,...,x_J)$ on $\Omega$ as \textit{multilinear} on an element $T\in \mathcal{T}_h$ if it is linear in each variable $x_{j} \in T_j$.  For concreteness and clarity of exposition, we will let $V_h$ be the space of piecewise multilinear functions on $\mathcal{T}_h$, although the results are applicable in other cases as well. That is, $$V_h = \Span \{\phi_i\}_{i=1}^N,$$ where each $\phi_i$ is defined as $$\phi_i(x_1,x_2,...,x_J) = \prod_{j=1}^J \phi^{(j)}_{i_j}(x_j),\qquad x_j\in \RR^{d_j}$$ for $\phi^{(j)}_{i_j}$ a piecewise-linear hat function on $\mathcal{T}_h^{(j)}$. 

Next, we associate to each vertex $x^{(j)}_i$ a polyhedral control volume $V^{(j)}_i$, such that together the $V_i^{(j)}$ partition $\Omega_j$, and we let $$V_i =  V_{i_1}^{(1)} \times V_{i_2}^{(2)}  \times ... \times V_{i_J}^{(J)}$$ for each $i = 1,2,...,N.$ We call the resulting subdivision $\{V_i\}_{i=1}^N$ of $\Omega$ the dual mesh of $\Omega$, and we call the $V_i$ voxels or dual mesh elements. From the viewpoint of the results we derive in the next section, it suffices for the voxels to obey the following two rules:
\begin{enumerate}
\item[1.] Each voxel $V_i$ contains a single vertex $x_i$ on its interior.
\item[2.]  \begin{align}\label{eq:voxel-identities}
    |V_i \cap T| = \frac{|T|}{N_T}1_{x_i\in T} \quad\text{so that}\quad |V_i| = \sum_{T \ni x_i} \frac{|T|}{N_T}.
\end{align}
\end{enumerate}
We note that these two properties imply that $T\cap V_i \neq \emptyset$ if and only if $x_i \in T$.  Moreover, for any function $f$ which is piecewise multilinear on $\mathcal{T}_h$, we have \begin{align*}
\int_\Omega f(x)dx = \sum_{T\in\mathcal{T}_h}   \frac{|T|}{N_T} \sum_{x_i\in T} f(x_i) &=
\sum_{T\in\mathcal{T}_h}   \sum_{x_i\in T} |V_i\cap T| f(x_i) \\& = \sum_{i=1}^N \sum_{T\ni x_i} |V_i \cap T| f(x_i) = \sum_{i=1}^N f(x_i)|V_i|.\end{align*} For any $f\in C(T)$, we refer to the quadrature rule \begin{equation}\label{eq:quad_mass_lump}
\int_{\Omega} f(x)dx \approx \mathcal{M}(f) := \sum_{i=1}^N f(x_i)|V_i| =: \sum_{T\in\mathcal{T}_h} \mathcal{M}_{h,T}(f).
\end{equation} resulting from the previous display as the \textit{mass lump quadrature rule}. 

\begin{remark}\label{remark:mass-lumping-to-voxel-sum}
After rewriting the righthand side of \eqref{eq:quadrature_scheme}  as \begin{align*}\sum_{i=1}^N u_iv_i \lambda_1(x_i)\int_{V_i} \lambda_2(x)1_K(x)dx &= \sum_{i=1}^N \sum_{T\in\mathcal{T}_h}  u_iv_i \lambda_1(x_i)\int_{V_i \cap T} \lambda_2(x)1_K(x)dx
\\&= \sum_{T\in\mathcal{T}_h}\sum_{x_i\in T} u_iv_i \lambda_1(x_i)\int_{V_i\cap T}   \lambda_2(x)1_K(x)dx,\end{align*} we see that \eqref{eq:quadrature_scheme} applies $\mathcal{M}_{h,T}$ on each element $T\in\mathcal{T}_h$ such that $\lambda_2$ and $1_K$ are constant on $T$ (e.g., $\partial K\cap \interior T = \emptyset$ and the splitting of $\lambda$ is chosen so that $\lambda_2 \equiv 1$). Similarly, if $\lambda_1$ and $1_K$ are constant on $T$, then we recover a finite-volume-type averaging quadrature rule.
\end{remark}

\begin{figure}[h]
    \centering
    \includegraphics[scale=.2]{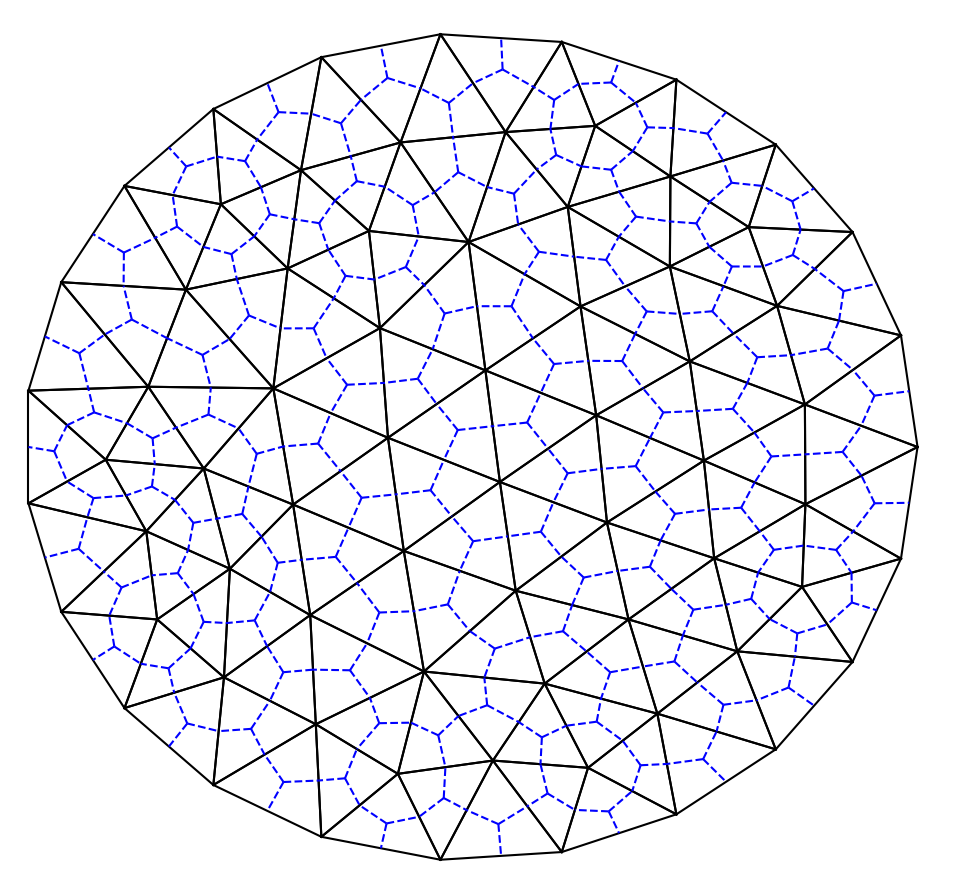}\hspace{10pt}\includegraphics[scale=.25]{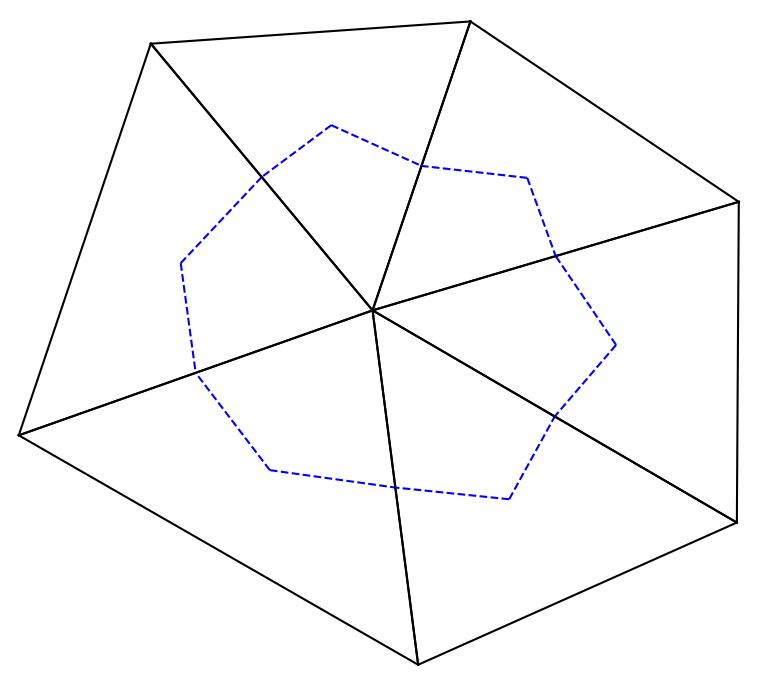}
    \caption{Left: A finite element primal mesh on a disk (black) and its dual (blue, dashed). Right: the supporting primal and dual mesh elements for an interior mesh node.}
    \label{fig:dual_mesh}
\end{figure}

In practice, we choose for the dual mesh of each $\mathcal{T}^{(j)}_h$ the barycentric dual (Definition \ref{def:barycentric_dual} below). In Definition \ref{def:barycentric_dual} and throughout the rest of the paper, we will use the notation $\conv A$ for the convex hull of a subset $A$ of a Banach space.

\begin{definition}\label{def:barycentric_dual}
    Let $\mathcal{S}_h$ be a triangulation of a domain $D\subseteq \RR^n$ into $n$-simplices and let $\{x_i\}_{i=1}^N$ be the nodes of $\mathcal{S}_h$. The \textit{barycentric dual} of $\mathcal{S}_h$ is the collection of voxels $\{V_i\}_{i=1}^N$ defined element-wise as $$V_i\cap T = \conv \{x : \textrm{$x$ is the barycenter of any face of $T$ containing $x_i$}\}.$$  Here a \textit{face} of $T$ is the simplex generated by any subset of vertices of $T$ (i.e., $T$ itself, along with, e.g., all edges and vertices of $T$ for $n = 2$, and all (geometric) faces as well for $n=3$). An example of a barycentric dual mesh (blue, dashed) resulting from a triangulation (black) of a disk domain is shown on the left in Figure \ref{fig:dual_mesh}, along with a close up of the primal and dual mesh elements supporting a single primal mesh node. 
\end{definition}

The barycentric dual mesh appears prominently in the finite volume and finite volume-element literature for structured and simplical unstructured meshes because it is convenient to compute and satisfies \eqref{eq:voxel-identities}. Both \cite{hackbusch_first_1989} and \cite{ewing_accuracy_2002}, for example, mention specifically the property \eqref{eq:voxel-identities}, and \cite{hackbusch_first_1989,ewing_accuracy_2002,liang_symmetric_2002,ma_symmetric_2003,xu_analysis_2009} all use the property mentioned in Remark \ref{remark:mass-lumping-to-voxel-sum} in their convergence analysis. The listed papers are focused on analyzing the finite-volume-element method as a perturbation of the finite element method, whereas we view \eqref{eq:quadrature_scheme} as a quadrature scheme which can be applied in an ad hoc way to particular terms of a PDE. In Remark \ref{remark:related_work} we discuss results from these works in more detail and compare them to the results we present here. 
    


\subsection{Function spaces and quadrature errors}\label{section:functionspacesandquadratureerrors}

Through the remainder of the paper, for any open set $D\subseteq \RR^d$ we will denote by $W^{k,p}(D)$  the usual Sobolev space consisting of functions whose derivatives of order $\leq k$ are in $L^p(D)$, with $H^k(D) = W^{k,2}(D)$. We use $\norm{\cdot}_{k,p,D}$ to denote the usual Sobolev space norm on $W^{k,p}(D)$ and $\norm{\cdot}_{k,D}$ the corresponding norm on $H^k(D)$. For the Sobolev space seminorm associated with $W^{k,p}(D)$ or $H^k(D)$, i.e.,  $$|f|_{k,p,K} = \left(\sum_{|\alpha| = k} \int_D |\partial^\alpha f|^p \right)^{\frac{1}{p}},$$ we write $|\cdot|_{k,p,D}$ and $|\cdot|_{k,D}$. Note, the definitions of the norms also apply for $k < 0$, in which case they refer to a dual norm or seminorm. We use $\left<\cdot,\cdot\right>_D$ to denote the $L^2$-inner product on $D$. Where there is no ambiguity in any of the above notations, we will occasionally drop the dependence on the domain $D$. The spaces $C^k(D)$ denote spaces of functions with continuous derivatives up to order $k$ on $D$. 

For any approximation (e.g., a quadrature scheme) $\mathcal{Q}(f)$ or $\mathcal{Q}(f,g)$ of a linear or bilinear functional (e.g., an integral), we can associate an error functional $\mathcal{E}(f)$ or bilinear form $\mathcal{E}(f,g)$ representing the difference between the exact value of the functional and the approximation. For convenience, we will occasionally overload notation, such that $\mathcal{E}(fg)$ and $\mathcal{E}(f,g)$ refer to the same error functional depending on whether we want to view $\mathcal{E}$ as a linear functional of $fg$ or a bilinear functional of $f$ and $g$. We also drop arguments of $\mathcal{E}$ that we want to view as fixed parameters. When necessary, we will use the notation $\mathcal{E}_D$ to denote the restriction of the error functional to a particular subset $D$ of the original domain of its arguments; for example, for the mass lump rule \eqref{eq:quad_mass_lump} we would have \begin{align*}
    \mathcal{E}(f) &= \bar{\mathcal{M}}_h(f) - \mathcal{M}_h(f) =  \int_{\Omega} f(x) dx - \sum_{T\in\mathcal{T}_h} \sum_{x_i\in\mathcal{T}} \frac{|T|}{N_T} f_{i},
\end{align*}
and, for any $T\in\mathcal{T}_h$, 
\begin{align*}
    \mathcal{E}_T(f) &= \int_{T} f(x) dx - \sum_{x_i\in T} \frac{|T|}{N_T} f_{i}.
\end{align*}
 For $T\in\mathcal{T}_h$, we refer to $\mathcal{E}_T(f)$ and $\mathcal{E}(f)$ as the local and global error functionals, respectively. The initial error estimates of Section \ref{section:error-estimates} will use the typical procedure of first bounding the local error on a \textit{reference element} $\hat{T}$ of unit size in terms of Sobolev space norms of its arguments. The reference element is chosen so that for each $T\in\mathcal{T}_h$, there exists an multi-affine (defined here in analogy with multilinear) transformation $G_T:\hat{T}\to T$. We refer to the reference element functional or bilinear form as $\hat{\mathcal{E}}$. To map a local error functional $\mathcal{E}_T$ onto the reference error functional $\hat{\mathcal{E}}$, we employ a change of variables which involves composing arguments $f$ of $\mathcal{E}$ with $G_T$. We denote a transformed function $f\circ G_T$ by $\hat{f}$. Since $G_T$ is multi-affine, if $f$ is multilinear on $T$ then $\hat{f}$ is multilinear on $\hat{T}$. By a scaling argument, if $f\in W^{k,p}(T)$ we have that \begin{equation}\label{eq:estimate_seminorm}
|\hat{f}|_{k,p,\hat{T}} \leq Ch^k |f|_{k,p,T}.    
\end{equation} Therefore, we gain powers of $h$ in our error estimates by applying the Bramble-Hilbert lemma on the reference element to change Sobolev space norms into semi-norms, and then scaling back to the original element. For convenience, we state a version of the Bramble-Hilbert lemma below which applies to both linear and bilinear forms. 

\begin{lemma}[Bramble-Hilbert for bilinear forms]\label{lemma:bramble_hilbert_bilinear}
Let $D\subseteq \RR^n$ have Lipschitz-continuous boundary. For some integer $k\geq 0$ and some $p\in [1,\infty]$, let $G$ be a continuous bilinear functional on $W^{k,p}(D) \times W^{\ell, q}(D)$ such that  \begin{align*}
    G(f,\cdot) &= 0 \quad \text{for all $f \in P_k$},\\
    G(\cdot,g) &= 0 \quad \text{for all $g \in P_\ell$}.
\end{align*} where $P_k$ denotes the set of polynomials of degree $\leq k$. Then there exists a constant $C(D)$ such that $$|G(f,g)| \leq C(D,G)|f|_{k+1,p,D}|g|_{\ell+1,q,D} \quad \text{for all $f\in W^{k,p}(D), g\in W^{\ell,q}(D)$}.$$
\end{lemma}

More details and background on the type of argument summarized above can be found in many classical finite element method references, including \cite{ciarlet_finite_2002}, Chapter 3 and 4.

We note that, while our quadrature error estimates do not require any mesh restrictions, the finite element convergence theorems proven in Section \ref{section:reactiondiffusionequations} require the existence of an interpolation operator $I_h: H^k(\Omega)\to V_h$ with error estimate \begin{equation}\label{eq:interpolation_operator}\left(\sum_{T\in\mathcal{T}_h} |v - I_hv|_{m,p,T}^p\right)^{\frac{1}{p}} \leq Ch^{k-m}|v|_{k,p,\Omega}.\end{equation} For this one typically requires a nondegeneracy condition \cite{scott_finite_1990}: \begin{equation}\label{eq:nondegeneracycondition}
\sup_{T\in\mathcal{T}_h} \frac{\diam T}{\rho_T} \leq \gamma_0,
\end{equation} where $\gamma_0$ is independent of $h$ and $\rho_T$ is the radius of the largest ball contained in $T$. The bound \eqref{eq:nondegeneracycondition} implies the existence of a constant $C$ such that the operator norm of the inverse transformation $G_T^{-1}$ is bounded by $Ch^{-1}$. The constant in \eqref{eq:interpolation_operator} depends on $\gamma_0$, among other quantities. 

\subsection{Error estimates}\label{section:error-estimates}


For any $u_h,v_h\in V_h$, define \begin{align*} \mathcal{E}(u_h, v_h; \lambda_1, \lambda_2, K) &= \int_\Omega \kappa(x)u_h(x)v_h(x)dx - \mathcal{Q}(u_h, v_h),
\end{align*} where $\mathcal{Q}$ is as in \eqref{eq:quadrature_scheme}. 


The proof of our first $O(h)$ consistency error estimate uses a relatively straightforward application of the Bramble-Hilbert lemma, the basic idea being that the quadrature is exact if $\lambda u_hv_h$ is a constant on each element $T$. We include a detailed version here for completeness. In later proofs, which follow a similar line of argument, we will omit some of the details included here. 

\begin{lemma}\label{lemma:H1consistency}
        Assume that $\lambda_1\in W^{1,\infty}(\Omega)$ and $\lambda_2 \in L^\infty(K)$. Then, there exists a constant $C$ such that for any $u_h,v_h\in V_h$, 
    \begin{enumerate}
        \item[(i)] For each $T\in\mathcal{T}_h$, the local quadrature error satisfies \begin{equation}\label{eq:H1localconsistencysmooth}
|\quaderrorlocal (\kappa u_h v_h) | \leq Ch \norm{\lambda_1}_{1,\infty,T}\norm{\lambda_2}_{0,\infty,T} \norm{u_hv_h}_{1,1,T},
\end{equation}
        \item[(ii)] The global quadrature error satisfies \begin{equation}\label{eq:H1globalconsistencysmooth}
|\quaderror (\kappa u_h v_h) | \leq Ch \norm{\lambda_1}_{1,\infty} \norm{\lambda_2}_{0,\infty}\norm{u_h}_{1}\norm{v_h}_{1}.
\end{equation}
    \end{enumerate}
\begin{proof}
    We first restrict to the local error functional for $T\in\mathcal{T}_h$: \begin{equation*}
        \quaderrorlocal(u_hv_h)  = \int_{T} \kappa(x)u_h(x)v_h(x)dx - \sum_{x_i \in T} u_iv_i\lambda_1(x_i) \int_{V_i \cap T} \lambda_2(x)1_K(x)dx,
    \end{equation*}
    and map to the reference element: 
    \begin{equation*}
        \quaderrorref(\hat{\lambda}_1 \hat{u}_h\hat{v}_h)  = \int_{\Tref} \hat{\kappa}(x)\hat{u}_h(x)\hat{v}_h(x)dx - \sum_{x_i \in T} u_iv_i\lambda_1(x_i) \int_{\widehat{V_i \cap T}} \hat{\lambda}_2(x)\hat{1}_K(x)dx.
    \end{equation*}
    Now, evidently we have the bound $$|\quaderrorref(f)| \leq C\norm{\lambda_2}_{0,\infty,\hat T}\norm{f}_{1,\infty,\hat T}$$ for any $f\in W^{1,\infty}(\Tref)$. Moreover, $\quaderrorref(f) = 0$ if $f$ is constant. Therefore, we can apply the Bramble-Hilbert lemma to obtain 
    \begin{align*}
     |\quaderrorref(\hat{\lambda}_1\hat{u}_h\hat{v}_h)| &\leq C\norm{\lambda_2}_{0,\infty,T}|\hat{\lambda}_1\hat{u}_h\hat{v}_h|_{1,\infty,T} \\&\leq C\norm{\lambda_2}_{0,\infty,T}(|\hat{\lambda}_1|_{0,\infty} |\hat{u}_h\hat{v}_h|_{1,\infty} + |\hat{\lambda}_1|_{1,\infty} |\hat{u}_h\hat{v}_h|_{0,\infty}) 
     \\& \leq  C\norm{\lambda_2}_{0,\infty,T}(|\hat{\lambda}_1|_{0,\infty} |\hat{u}_h\hat{v}_h|_{1,1} + |\hat{\lambda}_1|_{1,\infty} |\hat{u}_h\hat{v}_h|_{0,1})
    \end{align*} where in the last inequality we have used the equivalence of norms on a finite-dimensional space. Thus, mapping back to $T$ we apply \eqref{eq:estimate_seminorm} to obtain \begin{align*}
     |\quaderrorlocal(u_hv_h)| &\leq C|T||\quaderrorref^1(\hat{u}_h\hat{v}_h)| 
     \\&\leq C|T|\norm{\lambda_2}_{0,\infty,T}(|\hat{\lambda}_1|_{0,\infty} |\hat{u}_h\hat{v}_h|_{1,1} + |\hat{\lambda}_1|_{1,\infty} |\hat{u}_h\hat{v}_h|_{0,1})
 \\&\leq Ch\norm{\lambda_2}_{0,\infty,T}(|\lambda_1|_{0,\infty,T} |u_hv_h|_{1,1,T} + |\lambda_1|_{1,\infty,T} |u_hv_h|_{0,1,T})
 \\ &\leq Ch\norm{\lambda_2}_{0,\infty,T}\norm{\lambda_1}_{1,\infty,T}\norm{u_hv_h}_{1,1,T}
    \end{align*} which establishes \eqref{eq:H1localconsistencysmooth}. Summing over the elements then yields the desired global error bound \eqref{eq:H1globalconsistencysmooth}:
    \begin{align*}|\quaderror(u_hv_h)|& = \left|\sum_{T\in\mathcal{T}_h} \quaderrorlocal(u_hv_h) \right| \leq Ch\norm{\lambda_2}_{0,\infty,\Omega}\norm{\lambda_1}_{1,\infty,\Omega}\sum_{T\in\mathcal{T}_h} |u_hv_h|_{1,1,T} \\&= Ch\norm{\lambda_2}_{0,\infty,\Omega}\norm{\lambda_1}_{1,\infty,\Omega}|u_hv_h|_{1,1,\Omega}
    \\&\leq Ch\norm{\lambda_2}_{0,\infty,\Omega}\norm{\lambda_1}_{1,\infty,\Omega}\norm{u_h}_{1,\Omega}\norm{v_h}_{1,\Omega}.\end{align*}
\end{proof}
\end{lemma}

Our next result, Lemma \ref{lemma:L2consistency}, establishes a stronger local error estimate on elements that do not intersect the region of discontinuity. 

\begin{lemma} \label{lemma:L2consistency}
    Assume that $\lambda \in W^{1,\infty}(\Omega)$. Then, there exists a constant $C$ such that on each element $T$ such that $\interior T \cap \partial K = \emptyset$, \begin{equation}\label{eq:L2localconsistencysmooth}
|\quaderror_T (\kappa u_h v_h) | \leq Ch^2\norm{\lambda_2}_{1,\infty,T}\norm{\lambda_1}_{2,\infty,T}\norm{u_h}_{1,T}\norm{v_h}_{1,T} \quad \text{ for all } u_h,v_h\in V_h.  
    \end{equation}

\begin{proof}
    We observe that since $\interior T \cap \partial K = \emptyset$, $1_K(x) \equiv 0$ or $1_K(x) \equiv 1$ on $T$, i.e., $K\cap T = \emptyset$ or $T\subseteq K$. Evidently the error is zero if $K\cap T=\emptyset$, so we only consider the case $T\subseteq K$. We break the error functional $\quaderrorlocal$ into three parts: $$\quaderrorlocal(\kappa u_h v_h) = \quaderrorlocal^1(\lambda_2, \lambda_1 u_hv_h) + \quaderrorlocal^2(\lambda_1 u_hv_h; \lambda_2) + \quaderrorlocal^3(\lambda_2, \lambda_1 u_hv_h),$$ where
     \begin{equation}    \label{eq:smootherr1}
    \begin{aligned}
         \quaderrorlocal^1(\lambda_2, \lambda_1 u_hv_h) &= \int_T \lambda(x) u_h(x)v_h(x) dx \\&\qquad\qquad - \frac{1}{|T|}\left(\int_T \lambda_2(x)dx\right)\left(\int_T (\lambda_1u_hv_h)(x)dx\right)
   \end{aligned}
 \end{equation}
     \begin{equation}    \label{eq:smootherr2}
    \begin{aligned}
        \quaderrorlocal^2(\lambda_1 u_hv_h) &= \frac{1}{|T|}\left(\int_T \lambda_2(x)dx\right)\left(\int_T (\lambda_1u_hv_h)(x)dx\right) \\&\qquad\qquad-  \left(\int_T \lambda_2(x)dx\right)\left(\frac{1}{N_T}\sum_{x_i\in T} \lambda_1(x_i)u_{i}v_{i}\right)   
        \end{aligned}
 \end{equation}
      \begin{equation}    \label{eq:smootherr3}
    \begin{aligned}
        \quaderrorlocal^3(\lambda_2, \lambda_1 u_hv_h) &=\left(\int_T \lambda_2(x)dx\right)\left(\frac{1}{N_T}\sum_{x_i\in T}\lambda_1(x_i) u_{i}v_{i}\right) \\&\qquad\qquad-  \sum_{x_i\in T} \lambda_1(x_i)u_{i}v_{i}\int_{V_{i}\cap T} \lambda_2(x)dx.
        \end{aligned}
 \end{equation}

\begin{enumerate}
    \item[1.] To estimate $\quaderrorlocal^1$, we map to the reference element $\hat{T}$ to obtain the reference element error functional \begin{align*}
        \quaderrorref^1(\hat{\lambda}_2,\hat{\lambda}_1\hat{u}_h\hat{v}_h) &= \int_{\hat{T}} \hat{\lambda}(x) \hat{u}_h(x) \hat{v}_h(x) dx \\&\qquad\qquad- \frac{1}{|\hat{T}|}\left(\int_{\hat{T}} \hat{\lambda}_2(x)dx\right)\left(\int_{\hat{T}}(\hat{\lambda}_1\hat{u}_h\hat{v}_h)(x)dx\right).
    \end{align*} Now, we observe that $\quaderrorref^1$ is a bounded bilinear form on $W^{1,\infty}(\Tref) \times W^{1,\infty}(\Tref)$, since for any $f,g\in W^{1,\infty}(\Tref)$ $$\left|\hat{\mathcal{E}^1}(f,g)\right| \leq C\norm{f}_{1,\infty,\Tref}\norm{g}_{1,\infty,\Tref}.$$ Moreover, $\quaderrorref^1(f,g) = 0$ whenever either $f$ or $g$ is constant. It follows by the Bramble-Hilbert lemma for bilinear forms (Lemma \ref{lemma:bramble_hilbert_bilinear}) that there exists a constant $C = C(\Tref)$ such that \begin{equation}\label{eq:init_bound_e1_L2}
    \left|\quaderrorref^1(\hat{\lambda}_2,\hat{\lambda}_1\hat{u}_h\hat{v}_h)\right| \leq C|\hat{\lambda}_2|_{1,\infty,\Tref}|\hat{\lambda}_1\hat{u}_h\hat{v}_h|_{1,\infty,\Tref}.\end{equation} Now, applying the product rule and the equivalence of norms on a finite-dimensional space, we have that $$|\hat{\lambda}_1\hat{u}_h\hat{v}_h|_{1,\infty,\Tref} \leq C(|\hat{u}_h\hat{v}_h|_{0,1,\Tref}|\hat{\lambda}_1|_{1,\infty,\Tref} + |\hat{u}_h\hat{v}_h|_{1,1,\Tref}|\hat{\lambda}_1|_{0,\infty,\Tref}).$$ Substituting this bound into the estimate \eqref{eq:init_bound_e1_L2} and scaling between $T$ and the reference element, we have \begin{align*}
|\quaderrorlocal^1(\lambda_2, \lambda_1 u_hv_h)| &\leq C|T||\quaderrorref^1(\hat{\lambda}_2,\hat{\lambda}_1 \hat{u}_h\hat{v}_h)| \\&\leq C|T||\hat{\lambda_2}|_{1,\infty,\Tref}(|\hat{u}_h\hat{v}_h|_{0,1,\Tref}|\hat{\lambda}_1|_{1,\infty,\Tref} + |\hat{u}_h\hat{v}_h|_{1,1,\Tref}|\hat{\lambda}_1|_{0,\infty,\Tref}) \\&\leq Ch^2|\lambda_2|_{1,\infty,T}\norm{\lambda_1}_{1,\infty,T} \norm{u_hv_h}_{1,1,T}.
\end{align*}

\item[2.] The quadrature used on this step is a straightforward application of the mass-lump quadrature rule $\mathcal{M}_{h,T}$ to to evaluate $\int_T (\lambda _1 u_hv_h)(x) dx$. It follows from standard estimates (e.g., \cite{nie_lumped_1985}), following a similar argument to the one from the previous lemma, that 
\begin{align*}
|\quaderror^2(\lambda_1u_hv_h; \lambda_2)| &\leq Ch^2\norm{\lambda_2}_{0,\infty,T}\norm{\lambda_1}_{2,\infty,T}[\norm{u_hv_h}_{1,1,T} + |\diag\nabla^2(u_hv_h)|_{0,1,T}] 
\\&\leq Ch^2\norm{\lambda_2}_{0,\infty,T}\norm{\lambda_1}_{2,\infty,T}\norm{u_h}_{1,T}\norm{v_h}_{1,T},    
\end{align*} where $\diag \nabla^2 (u_hv_h)$ indicates the diagonal part of the matrix of second derivatives of $u_hv_h$. The second inequality uses the fact that $\diag \nabla^2 u_h = \diag \nabla^2 v_h = 0$, which is only true if $u_h$ and $v_h$ are multilinear.

\item[3.] To estimate \eqref{eq:smootherr3}, we again map to the reference element:
\begin{align*}
        \quaderrorref^3(\hat{\lambda}_2,\hat{\lambda}_1\hat{u}_h\hat{v}_h) &= \left(\int_{\Tref} \hat{\lambda}_2(x)dx\right)\left(\frac{1}{N_T}\sum_{x_i\in T} \lambda_1(x_i)u_{i}v_{i}\right)  \\&\qquad\qquad -   \sum_{x_i\in T} \lambda_1(x_i)u_{i}v_{i}\int_{\widehat{V_{i}\cap T}} \hat{\lambda}_2(x)dx,
\end{align*}
where we note that $|\widehat{V_i\cap T}| = \frac{\hat{T}}{N_T}$ by multilinearity of the volume transformation. 

Again, we have for any $f,g\in W^{1,\infty}(\Tref)$, that 
\begin{align*}
    |\quaderrorref^3(f, g)| \leq C\norm{f}_{1,\infty,\Tref}\norm{g}_{1,\infty,\Tref}.
\end{align*}
Moreover, $\quaderrorref^3(f, g)$ vanishes if $f$ or $g$ is constant, since if $g \equiv \bar{g}$ we have \begin{align*}\left(\int_{\Tref} f(x)dx\right)\left(\frac{1}{N_T}\sum_{x_i\in T} g_{i}\right) &= \bar{g}\int_{\Tref} f(x)dx \\&= \sum_{x_i\in T} \bar{g} \int_{\widehat{T \cap V_i}} f(x)dx = \sum_{x_i\in T} g_i \int_{\widehat{T \cap V_i}} f(x)dx,\end{align*} and on the other hand if $f \equiv \bar{f}$,
\begin{align*}\left(\int_{\Tref} f(x)dx\right)\left(\frac{1}{N_T}\sum_{x_i\in T} g_{i}\right) &= \sum_{x_i\in T} g_{i}\left(\bar{f} \frac{|\Tref|}{N_T}\right) =  \sum_{x_i\in T} g_{i} \left(\int_{\widehat{T \cap V_i}}f(x)dx\right).\end{align*} The last equality above follows from the fact that  $|\widehat{T \cap V_i}| = \frac{|\Tref|}{N_T}.$ Thus, following the same steps as in the estimate of \eqref{eq:smootherr1}, we obtain
\begin{align*}
|\quaderrorlocal(\lambda_2, \lambda_1 u_hv_h)| \leq Ch^2|\lambda_2|_{1,\infty,T}\norm{\lambda_1}_{1,\infty,T} \norm{u_h}_{1,T}\norm{v_h}_{1,T}.
\end{align*}
\end{enumerate}

Finally, to finish the proof we sum the error estimates to prove \eqref{eq:L2localconsistencysmooth}.
\end{proof}
\end{lemma}

\begin{remark}\label{remark:related_work}
    Proofs of various local consistency error estimates (in the $\lambda_1 \equiv 1$ case) related to Lemma \ref{lemma:L2consistency} have appeared in the finite-volume-element literature. We first mention \cite{hackbusch_first_1989}, who uses a dual mesh condition (Equation (2.4.1f) in \cite{hackbusch_first_1989}) similar to \ref{eq:voxel-identities} to estimate the difference between a finite-volume-element discretization of \eqref{eq:reaction-diffusion-parabolic-strong} and the corresponding $P^1$-finite element discretization in 2D. The proofs of the relevant estimates (Lemmas and Lemma 3.3.3) are valid only in two dimensions. A new proof appears in \cite{liang_symmetric_2002} (and similarly \cite{wu_error_2003}), and although the authors assume that $\Omega\subseteq \RR^2$ their proof apparently can generalize to multiple dimensions as well and is similar in spirit to our proof of Lemma \ref{lemma:L2consistency} (we mention also the related work \cite{ma_symmetric_2003}, which studies the parabolic reaction-diffusion equation). Finally, consistency estimates for the finite-volume-element discretization of the diffusion equation were explicitly generalized to multiple dimensions in \cite{xu_analysis_2009} (see also \cite{ewing_accuracy_2002}). The argument in \cite{ma_symmetric_2003} and \cite{xu_analysis_2009} goes as follows: if $f\in W^{1,p}(\Omega)$ and $v_h\in V_h$, then \begin{align*}
 \left|\int_{T} f(x) v_h(x) dx - \sum_{x_i\in T} v_h(x_i) \int_{V_i\cap T} f(x)dx  \right| &=  \left|\int_{T} [f(x) - \bar{f}_T] v_h(x) dx - \sum_{x_i\in T} v_h(x_i) \int_{V_i\cap T} [f(x) - \bar{f}_T]dx  \right|
 \\&\leq \left|f - \bar{f}_T\right|_{p,T}\left|v_h - \sum_{x_i\in T} v_h(x_i)1_{x\in V_i}\right|_{q,T} \leq Ch^2 \norm{f}_{1,p,T} \norm{v_h}_{1,q,T},       
    \end{align*} where $\frac{1}{p} + \frac{1}{q} = 1$ and $\bar{f}_T$ is the average value of $f$ on $T$. Lemma \ref{lemma:L2consistency} admits a similar proof, after estimating the interpolation error for $u_hv_h\lambda_2(x)$ onto $V_h$. For details we refer the reader to \cite{ma_symmetric_2003,xu_analysis_2009}. 
\end{remark}

\begin{remark}
In addition to its relationship with finite volume methods, the general idea of interpolating nonsmooth functions by averaging is related to the Cl\'ement and Scott-Zhang finite element interpolation operators \cite{clement_approximation_1975,scott_finite_1990}, which compute the projection $f_h$ of a given function $f$ onto a finite element space $V_h$ by first computing local $L^2$-projections $f_{h,i}$ of $f$ onto $V_h$ on a neighborhood $S_i$ of each node $x_i$, and then set $f_h(x_i) = f_{h,i}(x_i)$. We see that one can think of \eqref{eq:quadrature_scheme} as just one example from a family of local averaging methods for approximating integrals of nonsmooth functions.

To take this idea further, we note that a common choice for $S_i$ above is the nodal support of $x_i$, \begin{equation}\label{eq:nodal-support}
S_i =  \bigcup_{T\ni x_i} T.    
\end{equation} If we choose $V_i = S_i$ from \eqref{eq:nodal-support}, then the voxels now overlap (and \eqref{eq:voxel-identities} is not satisfied), but we can define a consistent quadrature scheme \begin{equation}\label{eq:support-averaging}
\int_\Omega u_h(x)v_h(x)\lambda_1(x)\lambda_2(x)1_K(x)dx \approx \sum_{i=1}^N \frac{u_iv_i}{N_T} \lambda_1(x_i)\int_{V_i} \lambda_2(x)1_K(x)dx,    
\end{equation} where the factor of $N_T$ reflects the fact that $$\sum_{T\ni x_i} \int_{V_i \cap T} \lambda_2(x)1_K(x)dx = N_T\int_T \lambda_2(x)1_K(x)dx.$$ With trivial modifications to our paper (especially Lemma \ref{lemma:L2consistency}) one could replace \eqref{eq:quadrature_scheme} with \eqref{eq:support-averaging} and arrive at the same convergence results. Although we only directly consider the barycentric dual here, it may prove useful in the future to consider other averaging schemes such as the one suggested above. Indeed, \eqref{eq:support-averaging} may be attractive if only because the nodal support is typically easily accessible in an unstructured mesh finite element implementation, and so the barrier to implementation is lower than for \eqref{eq:quadrature_scheme}.
\end{remark}

\begin{remark}\label{remark:suboptimalerrorestimates}
The error bounds of Lemma \ref{lemma:H1consistency} and \ref{lemma:L2consistency} involve $L^\infty(\Omega)$ and $W^{1,\infty}(\Omega)$ norms of $\lambda_2$, respectively, compared with $W^{1,\infty}(\Omega)$ and $W^{2,\infty}(\Omega)$ norms of $\lambda_2$. Therefore, we would expect that, all else being equal, choosing $\lambda_1 \equiv 1$ in \eqref{eq:quadrature_scheme} (which we refer to later as the averaging method) should result in smaller quadrature errors compared with $\lambda_2\equiv 1$ (referred to as the lumping method), especially as the regularity of $\lambda$ deteriorates. We briefly investigate this idea numerically in Section \ref{section:numerics}, but do not observe significant differences in the errors of the two methods. While we do not claim the error bounds computed here are optimal, it does seem reasonable to expect that the averaging approximation of $\lambda_1$ would require less regularity than the pointwise evaluation required by the lumping method, particularly in higher dimensions. We leave a more thorough investigation of this aspect of the method to future work.
\end{remark}

\section{Refined error bounds on the interface}\label{section:refinedconvergenceresults}

Let \begin{align*}
    \mathcal{X}_{K,h} &= \{T \in \mathcal{T}_h : \interior T \cap K \neq \emptyset\},\\
    K_h &= \bigcup_{T \in \mathcal{X}_{K,h}} T.
\end{align*}

Then combining the local error estimates from Lemma \ref{lemma:H1consistency} and \ref{lemma:L2consistency}, we obtain the following global error estimate:

\begin{corollary}\label{corollary:initial_global_error_estimate}
    For any $u_h,v_h\in V_h$, 
    \begin{equation}\label{eq:initialglobalerrestimate}
    \begin{aligned}
        |\quaderror(u_hv_h)| &\leq Ch^2\sum_{T\in \mathcal{T}_h\setminus\mathcal{X}_{K,h} } Ch^2\norm{u_h}_{2,T}\norm{v_h}_{2,T} + Ch\sum_{T\in \mathcal{X}_{K,h} } Ch|u_hv_h|_{1,1,T}\\
        &= Ch^2\sum_{T\in \mathcal{T}_h\setminus\mathcal{X}_{K,h} } \norm{u_h}_{2,T}\norm{v_h}_{2,T} + Ch|u_hv_h|_{1,1,K_h}.
    \end{aligned}
    \end{equation}
\end{corollary}

Our object in this section will be to improve upon \eqref{eq:initialglobalerrestimate} by studying the term $|u_hv_h|_{1,1,K_h}$. The $O(h)$ estimate of Lemma \ref{lemma:H1consistency} for the quadrature errors is the best possible on each element $T\in\mathcal{X}_h$, since without a mesh restriction the intersection of $K$ with such elements can be arbitrary. On the other hand, if we assume that $|K_h| = O(h)$, then \begin{equation}\label{eq:Linferrorestimate}
    |u_hv_h|_{1,1,K_h} \leq C|K_h|  \norm{u_h}_{1,\infty}\norm{v_h}_{1,\infty} \leq Ch \norm{u_h}_{1,\infty}\norm{v_h}_{1,\infty}.
\end{equation} Heuristically, the idea is that the larger $O(h)$ errors local to the discontinuity contribute relatively little to the global error. However, \eqref{eq:Linferrorestimate} is not very useful for finite element error analysis, which normally requires an integral norm on $v_h$. The following related error estimate will be used to prove a supercloseness result for reaction-drift-diffusion equations:

\begin{theorem}\label{thm:superconvergence_lemma}
    Assume that $|K_h| = O(h)$. Then, for any $2\leq p \leq \infty$ there exists a constant $C$ such that $$|\mathcal{E}(u_hv_h)| \leq Ch^{\frac{3p - 2}{2p}} \norm{u_h}_{1,p}\norm{v_h}_1.$$
\end{theorem}

Next, we turn our attention to $O(h^2)$ error estimates involving integral norms of $v_h$.  In particular, the remainder of this section will be devoted to replacing the $W^{1,\infty}(\Omega)$ norms in \eqref{eq:Linferrorestimate} with $H^2(\Omega)$ norms. The analysis is based on the following mild assumption about the smoothness of $K$ (which, incidentally, also implies $|K_h| = O(h)$): 

\begin{assumption}\label{assumption:lipschitz_boundary_assumption}
    There is a Lipschitz domain $\hat{K}\subseteq \RR^d$ such that $(\partial K \cap \Omega) \subseteq \partial \hat{K}$.  By Lipschitz domain, we mean that $\hat{K}$ is open and that for each $x \in \partial \hat{K}$, there exists an open hypercube $V\subseteq  \RR^{d-1}$, a Lipschitz function $\varphi : V\to \RR$, two constants $\alpha,\beta\in\mathbb{R}$ such that $\alpha < \varphi(y) < \beta$ for all $y\in V$, and a linear rotation map $Q:\RR^{d}\to\RR^d $ such that \begin{enumerate}
        \item $Qx\in V\times [\alpha,\beta]$ and
        \item $Q\hat{K} \cap [V\times (\alpha,\beta)] = \{ (y,w) : y\in\RR^{d-1}, \alpha < w < \varphi(y)\}$.
    \end{enumerate}
    In other words, there is a rotation of $\RR^d$ described by $Q$ such that $\partial \hat{K}$ can be locally written as the graph of a Lipschitz function $\varphi$ on some open cube $V$.
\end{assumption}

\begin{remark}
 Assumption \ref{assumption:lipschitz_boundary_assumption} essentially requires that the portion of $\partial K$ interior to $\Omega$ is a Lipschitz curve. We chose this slightly more general statement in lieu of the more straightforward assumption that $K$ itself is Lipschitz in part to handle the case where $K$ is \textit{defined} as the intersection of a larger set $\hat{K}$ with $\Omega$, where $\hat{K}$ has a Lipschitz boundary. The distinction is necessary, since even if $\Omega$ and $\hat{K}$ each have smooth boundaries, $K = \hat{K} \cap \Omega$ can still have cusps. The more general assumption may also be useful in the case where $\partial K$ is defined by a Lipschitz level set function that can intersect $\partial \Omega$, but information about that intersection is not known a priori. 
 
 Of course, if $K$ itself is known to be Lipschitz, the natural choice is to take $\hat{K} = K$ (in particular, if $K$ is compactly contained in $\Omega$). We note, however, that even if $K$ is Lipschitz, if it shares part its boundary with $\Omega$ then the constants in the error analysis may improve by choosing $\hat{K}$ to be different from $K$ (the extreme example being $K = \Omega$, $\hat{K} = \emptyset$).
    
\end{remark}

 We start our analysis by defining the distance function $d_D: \RR^d\to\RR$ for any set $D\subseteq \RR^d$ by $$d_D(x) = d(x, D) = \inf_{y\in D} |x - y|.$$  Our object will be to obtain uniform (in $h$) control over the traces of $H^1(\Omega)$ functions on the $d_K$ and $d_{K^c}$-level sets (here $K^c$ is the complement of $K$ taken relative to $\Omega$) $$\Gamma_\varepsilon = \{x \in K : d_{K^c}(x) = \varepsilon\}\qquad \textrm{and}\qquad\tilde{\Gamma}_\varepsilon = \{x \in K^c: d_K(x) = \varepsilon\}.$$ In Subsection \ref{section:uniformtraceinequality} below, we use the differentiability properties of $d_{\hat{K}}$ and $d_{\hat{K}^c}$ to characterize the normal vectors of their level sets, which contain $\Gamma_h$ and $\tilde{\Gamma}_h$.
 
 \subsection{Uniform trace inequalities}\label{section:uniformtraceinequality}
 
In this subsection, we will establish several facts about general Lipschitz domains $D$, that we later apply to $\hat{K}$. First, the normal vectors of $D$ can be directly related to the constant appearing in the trace inequality for $D$ via the following elementary lemma, appearing in \cite{grisvard_elliptic_2011}:
 
\begin{lemma}[Characterization of trace constant]\label{lemma:characterizationoftraceconstant}
    Let $D$ be a Lipschitz domain. Then 
      \begin{enumerate}
        \item There exists a vector field $\bmu_D\in W^{1,\infty}(\RR^d; \RR^d)$ and constant $\delta_D > 0$ such that \begin{equation}\label{eq:traceineqnormal}
        \bmu_D(x) \cdot \bn_D(x) > \delta_D \quad \textrm{a.e. $x\in\partial D$},     
        \end{equation} 
        where $\bn_D$ is the outward pointing normal on $\partial D$.
        \item The constant $C_D$ in the trace inequality for $D$ is bounded by $(\delta_D)^{-1}\norm{\bmu_D}_{1,\infty,D}$ for any vector field $\bmu_D \in W^{1,\infty}(\RR^d; \RR^d)$ satisfying \eqref{eq:traceineqnormal}. That is, there is a constant $C(p)>0$ independent of $D$ such that for any $f\in W^{1,p}(D)$, we have $$\delta_D \norm{f}_{0,p,\partial D} \leq C(p)\norm{\bmu_D}_{1,\infty,D} \norm{f}_{1,p,D}.$$
    \end{enumerate}

\begin{proof}
    This is essentially Lemmas 1.5.1.9 and 1.5.1.10 from \cite{grisvard_elliptic_2011}, except that we replace their $C^\infty(D; \RR^d)$ vector field with a $W^{1,\infty}(\RR^d; \RR^d)$ extension.
\end{proof}
\end{lemma}

To state the next lemma, we introduce the projection $P_D:\RR^d\to \overline{D}$ as $$P_D(x) = \argmin_{y\in \overline{D}} |x - y|.$$ Note that $P_D$ is only well-defined at almost every point $x\in \overline{D}^c$, since the set $\{y \in \overline{D} : |x - y| = d_D(x)\}$ may have more than one member. At such $x$, we will show that the direction vector $x - P_D(x)$ from $x$ to $P_D(x)$ can be written as a nonnegative linear combination of normal vectors $\bn_D$ of $D$ infinitesimally close to $P_D(x)$. Crucially, this relationship holds even when $\bn_D$ is not defined at $P_D(x)$ --- although $\bn_D$ is well-defined a.e. with respect to the $\partial D$ surface measure, $\bn_D$ may be ill-defined at $P_D(x)$ on an open subset of $D^c$. If $D$ is a $C^2$ domain, then $P_D$ is well-defined for $\varepsilon$ sufficiently small and $\frac{x - P_D(x)}{|x - P_D(x)|} = \bn_D(P_D(x))$, a result that has already been proven in many different sources (see, e.g., \cite{gilbarg_elliptic_2001}, Theorem 14.16). The novelty of Lemma \ref{lemma:characterizationofminimizers} stems from the weakening of the $C^2$ hypothesis.

\begin{lemma}\label{lemma:characterizationofminimizers}
Let $D$ be a Lipschitz domain with outward pointing normal $\bn_D$. If $x\notin \overline{D}$ and the projection $P_D(x)$ of $x$ onto $D$ is well-defined, then $$\frac{x - P_D(x)}{|x - P_D(x)|} = \frac{\sum_{j=1}^m \alpha_j \lim_{i\to\infty} \bn_D(x_i^{(j)})}{\left|\sum_{j=1}^m \alpha_j \lim_{i\to\infty} \bn_D(x_i^{(j)})\right|}$$ where $\alpha_j \geq 0$ and each $x_i^{(j)}\in \partial D$ is a sequence such that $x_i^{(j)}\to P_D(x)$ and $\bn_D(x_i^{(j)})$ converges as $i\to\infty$.

\begin{proof}
    See \hyperref[section:appendix]{Appendix}.
\end{proof}
\end{lemma}

Before proving the uniform trace inequality, we collect a few more basic properties of $d_D$, which help tie Lemmas \ref{lemma:characterizationoftraceconstant} and \ref{lemma:characterizationofminimizers} together by characterizing the normal vectors $\bn_\varepsilon(x)$ to the $\varepsilon$-level sets of $d_D$ as being almost-everywhere parallel to the direction vector $x - P_D(x)$. Define the signed distance function $\rho:\RR^d \to \RR$ by $$\rho(x) = \begin{cases}
    d_D(x) & x \in D^c \\
    -d_{D^c}(x) & x \in D.
\end{cases}$$
Further define, for all $\varepsilon > 0$,
$$U_\varepsilon = \{x\in \RR^d : \rho(x) < \varepsilon\},\qquad L_\varepsilon = \{x\in\RR^d : \rho(x) < -\varepsilon\}$$ so that $$\partial U_\varepsilon = \{x\in \RR^d : \rho(x) = \varepsilon\},\qquad  \partial L_\varepsilon = \{x\in\RR^d : \rho(x) = -\varepsilon\}.$$ According to \cite{doktor_approximation_1976}, Theorem 4.1, $U_\varepsilon$ and $L_\varepsilon$ are Lipschitz domains for $\varepsilon$ sufficiently small, and therefore have a.e. well-defined normal vectors. First considering $U_\varepsilon$, we can characterize those normal vectors as follows:
\begin{enumerate}
\item $\rho$ is differentiable at $x\in \overline{D}^c$ if and only if the projection $P_D(x)$ is single-valued (see, e.g., \cite{fitzpatrick_metric_1980}); at any such point $x$ we have $\nabla d_D(x) = \frac{x - P_D(x)}{|x - P_D(x)|}$.
\item Since $d_D$ is Lipschitz, the set of points where $d_D$ is not differentiable is measure zero in $\RR^d$; it follows that for almost every $\varepsilon > 0$, $d_D$ is differentiable almost everywhere on the $\varepsilon$ level set of $\rho$, $\partial U_\varepsilon$.
\item  At such $\varepsilon$, we conclude from (1) and (2) that the outward pointing normal vector $\bn_\varepsilon(x)$ for $U_\varepsilon$ coincides almost everywhere with $\frac{x - P_D(x)}{|x - P_D(x)|}$.
\end{enumerate}
Following a similar argument, we can show that the outward pointing normal vector of $L_\varepsilon$ coincides almost everywhere with $\frac{P_{D^c}(x) - x}{|P_{D^c}(x) - x|}$.

\begin{lemma}[Uniform trace inequality] \label{lemma:uniformtraceinequality} 
    There exists $\varepsilon_0 > 0$, such that for almost every $\varepsilon < \varepsilon_0$, the constants $C_{U_\varepsilon}$ and $C_{L_\varepsilon}$ in the trace inequalities $$\left(\int_{\partial U_\varepsilon} |f|^p dS(x)\right)^{\frac{1}{p}} \leq C_{U_\varepsilon}(p)\norm{f}_{1,p,U_\varepsilon}\qquad \left(\int_{\partial L_\varepsilon} |f|^p dS(x)\right)^{\frac{1}{p}} \leq C_{L_\varepsilon}(p)\norm{f}_{1,p,L_\varepsilon}$$ can be chosen independent of $\varepsilon$. That is, there exists $C_{\varepsilon_0}(p)$ such that $C_{U_\varepsilon},C_{L_\varepsilon} \leq C_{\varepsilon_0}(p)$ for almost every $\varepsilon < \varepsilon_0$. Moreover, by choosing $\varepsilon_0$ sufficiently small, $C_{\varepsilon_0}$ can be chosen arbitrarily close to $\frac{\norm{\bmu}_{1,\infty,\RR^d}}{\delta}$, where $\bmu = \bmu_D$ and $\delta = \delta_D$ are any vector field and positive constant satisfying \eqref{eq:traceineqnormal} from Lemma \ref{lemma:characterizationoftraceconstant}.
    \begin{proof}

    We will prove the inequality for $U_\varepsilon$, since the proof for $L_\varepsilon$ is completely analogous using the result of Lemma \ref{lemma:characterizationoftraceconstant} as below and applying Lemma \ref{lemma:characterizationofminimizers} to $\overline{D}^c$.
    
    To prove the uniform trace theorem for $U_\varepsilon$, we will use Lemma \ref{lemma:characterizationoftraceconstant} to show that there exists $\varepsilon_0$ such that, for almost every $\varepsilon < \varepsilon_0$, $\bmu$ and $\delta$ satisfy the inequality from \eqref{eq:traceineqnormal} for almost every $x \in U_\varepsilon$. To that end, we choose $\varepsilon_0$ such that $U_\varepsilon$ is Lipschitz for $\varepsilon < \varepsilon_0$ and so that $\varepsilon_0 < \frac{\delta}{|\bmu|_{1,\infty,\RR^d}}$.
    
    Let $\varepsilon < \varepsilon_0$ be such that $\bn_\varepsilon(x) = \frac{x - P_D(x)}{|x - P_D(x)|}$ for almost every $x\in \partial U_\varepsilon$, and let $x\in \partial U_\varepsilon$ be such a point. Then by Lemma \ref{lemma:characterizationofminimizers}, there exist sequences $x^{(j)}_i\to P_D(x)$ and constants $\alpha_j \geq 0$, $j=1,2,...,m$, such that $$\bn_\varepsilon(x)  = \frac{\sum_{j=1}^m \alpha_j \lim_{i\to\infty} \bn_D(x_i^{(j)})}{\left|\sum_{j=1}^m \alpha_j \lim_{i\to\infty} \bn_D(x_i^{(j)})\right|},$$ Using the continuity of $\bmu$, we have that $\lim_{i\to\infty} \bmu(P_D(x)) \cdot \bn_D(x_i^{(j)}) > \delta$, so the triangle inequality implies $$\bmu(P_D(x))\cdot \bn_\varepsilon(x) = \frac{\sum_{j=1}^m \alpha_j \lim_{i\to\infty} \bmu(P_D(x)) \cdot \bn_D(x_i^{(j)})}{\left|\sum_{j=1}^m \alpha_j \lim_{i\to\infty} \bn_D(x_i^{(j)})\right|} > \delta.$$ Thus finally \begin{align*}
\bn_\varepsilon(x) \cdot \bmu(x) &= \bn_\varepsilon(x) \cdot \bmu(P_D(x)) + \bn_\varepsilon(x)\cdot [\bmu(x) - \bmu(P_D(x))]  \\&> \delta - \varepsilon_0 |\bmu|_{1,\infty,\RR^d} > 0.      
    \end{align*}
     Hence, applying Lemma \ref{lemma:characterizationoftraceconstant} the constant in the trace inequality can be chosen uniformly as $C_{\varepsilon_0}(p) = C(p)\frac{|\bmu|_{1,\infty,\RR^d}}{\delta - \varepsilon_0 |\bmu|_{1,\infty,\RR^d}}$ for $\varepsilon < \varepsilon_0$. As $\varepsilon_0 \to 0$, the upper bound $C(p)\frac{|\bmu|_{1,\infty,\RR^d}}{\delta - \varepsilon_0 |\bmu|_{1,\infty,\RR^d}}$ converges to $ C(p)\frac{|\bmu|_{1,\infty,\RR^d}}{\delta},$ as claimed.
    \end{proof}
\end{lemma}

\begin{remark}
    In contrast to our pointwise result in Lemma \ref{lemma:characterizationofminimizers}, Theorem 4.1 of \cite{doktor_approximation_1976} gives an $L^p$, $p < \infty$, convergence result relating the normal vectors of a Lipschitz domain $D$ to those of $U_\varepsilon$ as $\varepsilon\to 0$. To directly use that result would have been an alternative approach to proving Lemma \ref{lemma:uniformtraceinequality}. 
\end{remark}

\subsection{A refined error bound}

Finally, with the help of Lemma \ref{lemma:uniformtraceinequality}, we can now improve the error bound near $\partial K$ in the situation when $u_h$ and $v_h$ are $V_h$ approximations of functions in $H^2(\Omega)$. First, we show how the H\"older inequality employed for Theorem \ref{thm:superconvergence_lemma} can be formally replaced with a trace inequality, trading off derivatives instead of powers of $p$.

\begin{proposition}\label{proposition:controlonKh}
Assume that $f\in W^{1,p}(\Omega)$. Then, there exists $h_0$ and $C = C(h_0, p) > 0$ such that for all $h < h_0$, $$|f|_{0,p,K_h} \leq Ch^{\frac{1}{p}}\norm{f}_{1,p,\Omega}.$$

\begin{proof}
    In the following, we define $\rho$, $U_\varepsilon$, and $L_\varepsilon$ as in the previous subsection with $D = \hat{K}$. Observe that for each $h$, $$K_h\subseteq \bigcup_{-h < \varepsilon < h} \{\rho(x) = \varepsilon\} = \bigcup_{0 \leq \varepsilon < h} \partial U_\varepsilon \cup \partial L_\varepsilon =: X_h.$$ Let $h_0$ be a constant such that the uniform trace inequality of Lemma \ref{lemma:uniformtraceinequality} applies on $U_\varepsilon$ and $L_\varepsilon$ for $h < h_0$, and let $\ext{\Omega} =\Omega \cup X_{h_0}$. By the Sobolev extension theorem, there exists an extension $\ext{f} \in W^{1,p}(\ext{\Omega})$ of $f$ such that $\norm{\ext{f}}_{1,p,\ext{\Omega}} \leq C\norm{f}_{1,p,\Omega}$, where $C$ is independent of $f$. Therefore, noting that $\rho$ is Lipschitz with constant 1, the co-area formula implies that $$|f|_{0,p,K_h} \leq |\ext{f}|_{0,p,X_h} = \left(\int_{-h}^h \int_{\{\rho = \varepsilon\}} |\ext{f}|^p dS_\varepsilon d\varepsilon\right)^{\frac{1}{p}}.$$ Applying the uniform trace inequality on $\{\rho = \varepsilon\}$ makes the integrand independent of $h$, so that $$\left(\int_{-h}^h \int_{\{\rho = \varepsilon\}} |\ext{f}|^p dS_\varepsilon d\varepsilon\right)^{\frac{1}{p}} \leq C(h_0,p)\left(\int_{-h}^h\norm{\ext{f}}_{1,p,\ext{\Omega}}^p d\varepsilon\right)^{\frac{1}{p}} \leq Ch^{\frac{1}{p}} \norm{f}_{1,p,\Omega},$$ as needed.
\end{proof}
\end{proposition}

\begin{remark}
We mention here a few results related to Proposition \ref{proposition:controlonKh} that have appeared in previous literature, primarily in the context of finite element error analysis for piecewise linear approximation of a smooth domain $\Omega$. Specifically, a similar result for $d=2$ is stated as Lemma 2.1 in \cite{zhang_analysis_2008}, and later \cite{kashiwabara_penalty_2016} gives a detailed proof for arbitrary $d$ in their Theorem 8.3. The authors of \cite{kashiwabara_penalty_2016} assume that the domain is $C^{1,1}$, however, and their proof uses facts specific to that case. 
\end{remark}
\begin{theorem}\label{thm:boundary_skin_estimate}
    Assume that $K$ satisfies Assumption \ref{assumption:lipschitz_boundary_assumption}. Then, there exists $h_0 > 0$ such that whenever $h < h_0$, for every $u_h,w_h\in V_h$ and $u,w\in H^2(\Omega)$ we have \begin{align*}
|u_hw_h|_{1,1,K_h} \leq C[\norm{u - u_h}_{1,K_h}\norm{w - w_h}_{1,K_h} &+ h^{\frac{1}{2}}\norm{u}_{2,\Omega}\norm{w - w_h}_{1,K_h} \\&+ h^{\frac{1}{2}}\norm{w}_{2,\Omega}\norm{u - u_h}_{1,K_h} +  h\norm{u}_{2,\Omega} \norm{w}_{2,\Omega}].        
    \end{align*} 
      \begin{proof}
    The Cauchy-Schwarz inequality implies that $$|u_hw_h|_{1,1,K_h} \leq C\norm{u_h}_{1,K_h}\norm{v_h}_{1,K_h}.$$ Then, we have that \begin{align}
    \norm{u_h}_{1,K_h}\norm{v_h}_{1,K_h} &\leq (\norm{u - u_h}_{1,K_h} + \norm{u}_{1,K_h})(\norm{w - w_h}_{1,K_h} + \norm{w}_{1,K_h})\label{eq:initialestimateonKh}
\end{align} 
For $h$ sufficiently small, we can apply Proposition \ref{proposition:controlonKh} to $u,w$, and their partial derivatives $\frac{\partial u}{\partial x_i},\frac{\partial w}{\partial x_i}$ for $i = 1,...,d$, to obtain $$\norm{u}_{1,K_h} \leq Ch^{\frac{1}{2}} \norm{u}_{2,\Omega},\qquad \norm{w}_{1,K_h} \leq Ch^{\frac{1}{2}} \norm{w}_{2,\Omega}.$$ Combining this estimate with \eqref{eq:initialestimateonKh} yields the result.
\end{proof}
\end{theorem}

Thus we have the following corollary, which combines Corollary \ref{corollary:initial_global_error_estimate} and Theorem \ref{thm:boundary_skin_estimate}:

\begin{corollary}\label{corollary:final_global_error_estimate}
    Let $u,w\in H^2(\Omega)$ and assume $u_h,w_h\in V_h$ satisfy the asymptotic error estimates \begin{equation}\label{eq:asympotic_error_estimates}
        \norm{u - u_h}_1 \leq Ch|u|_2,\qquad \norm{w - w_h}_1\leq Ch|w|_2.
    \end{equation}
    Then the quadrature error for \eqref{eq:quad_mass_lump} applied to $u_h,w_h$ satisfies the aymptotic quasi-optimal global error estimate
    \begin{equation}\label{eq:finalglobalerrestimate}
    \begin{aligned}
        |\quaderror(u_hw_h)| &\leq Ch^2\norm{u}_{2,\Omega}\norm{w}_{2,\Omega}.
    \end{aligned}
    \end{equation}
\end{corollary}

\section{Convergence results and monotonicity for reaction-drift-diffusion equations}\label{section:reactiondiffusionequations}

Starting from the weak form of \eqref{eq:reaction-diffusion-parabolic-strong}, we denote the exact solution $u\in H^1_D(\Omega)$ as the unique solution to \begin{equation}\label{eq:reaction-diffusion-parabolic-weak}
\begin{aligned}
   a(u,v) &= -r(u,v) + \left<f, v\right>, \quad \text{for all }v\in H^1_D(\Omega),
\end{aligned}
\end{equation} where $H^1_D(\Omega)$ is the subspace of functions in $H^1(\Omega)$ with zero trace on $\Gamma_D$. We assume for simplicity homogeneous boundary conditions (i.e., $g_D,g_N \equiv 0$).

We will assume that the mesh $\mathcal{T}_h$ is aligned with $\Gamma_D$ and $\Gamma_N$, so that we can incorporate the Dirichlet conditions into $V_h$ by removing the basis functions with nodes touching the Dirichlet boundary. Then, we can define the Galerkin finite element solution $u_h^G$ associated with $V_h$ as the solution of reaction-diffusion equation  \begin{equation}\label{eq:reaction-diffusion-parabolic-weak-fem}
\begin{aligned}
   a(u_h^G,v_h) &= -r(u_h^G,v_h) + \left<f, v_h\right>, \quad \text{for all }v_h\in V_h,
\end{aligned}
\end{equation}where $\left<\cdot,\cdot\right>$ is the $L^2$ inner product on $\Omega$. 

Finally, our quadrature approximation to \eqref{eq:reaction-diffusion-parabolic-strong} is given by
 \begin{equation}\label{eq:reaction-diffusion-parabolic-weak-fem-quadrature}
\begin{aligned}
     a_h(u_h,v_h) &= -r_h(u_h,v_h) + \left<f,v_h\right> , \quad \text{for all }v_h\in V_h
\end{aligned}
\end{equation} 
where $a_h(u_h,v_h)$ is a quadrature approximation to the drift-diffusion operator. To state error estimates for the quadrature approximation \eqref{eq:reaction-diffusion-parabolic-strong}, we will need to assume that the discrete drift-diffusion operator $a_h$ satisfies similar approximation properties to $r_h$. In particular, we assume that 
\begin{equation}\label{eq:assumptions_on_discrete_transport_operator}
\begin{aligned}
    |a(w_h,v_h) - a_h(w_h,v_h)| &\leq Ch\norm{w_h}_1\norm{v_h}_1 \quad &&\text{for all $w_h,v_h\in V_h$}, \\
    |a(w_h,v_h) - a_h(w_h,v_h)| &\leq Ch^{\frac{3p-2}{2p}}\norm{w_h}_{1,p}\norm{v_h}_1 \quad &&\text{for all $w_h,v_h\in V_h$},  \\
    |a(u_h,w_h) - a_h(u_h,w_h)| &\leq Ch^2\norm{u}_2\norm{w}_2, 
\end{aligned}
\end{equation}
where in the last line $u_h$ and $w_h$ again satisfy the asymptotic error estimates \eqref{eq:asympotic_error_estimates}.

We also assume that $\alpha$ and $\bbeta$ are continuous on $\overline{\Omega}$, and in addition that $\alpha$ is positive definite. These conditions are not necessarily optimal, and in future work we hope to relax our assumptions somewhat to allow discontinuous or singular coefficients $\bbeta$. However, for our purposes in this work the current assumptions are enough to imply the coercivity condition for $w_h\in V_h$:
\begin{align}\label{eq:inf-sup}
    \sup_{v_h\in V_h} \frac{a_h(w_h,v_h) + r_h(w,v_h)}{\norm{v_h}_1} > c_0\norm{w_h}_1
\end{align}
where $c_0 > 0$ is independent of $h$. Therefore, there is a unique solution each problem associated with \eqref{eq:reaction-diffusion-parabolic-weak-fem-quadrature}. 

Using the convergence estimates of Corollaries \ref{corollary:initial_global_error_estimate} and \ref{corollary:final_global_error_estimate}, the error analysis for \eqref{eq:reaction-diffusion-parabolic-weak-fem-quadrature} is fairly routine. We give convergence and superconvergence results in the $H^1(\Omega)$ norm in Theorem \ref{thm:H1_reaction_diffusion_convergence}, and an $L^2(\Omega)$ convergence proof in Theorem \ref{thm:abstract_L2_convergence_theorem}.

\begin{theorem}\label{thm:H1_reaction_diffusion_convergence}
Given Assumption \ref{assumption:lipschitz_boundary_assumption} and the inf-sup condition \ref{eq:inf-sup}, there exists $C > 0$ such that $$\norm{u^G_h - u_h}_1 \leq Ch\norm{u}_1,\qquad \norm{u - u_h}_1 \leq C\inf_{v_h\in V_h} \norm{v_h - u}_1 +Ch\norm{u}_1.$$ If $u\in W^{1,p}(\Omega)$ for some $2 < p \leq \infty$ and there exists an interpolation operator $I_h$ satisfying \eqref{eq:interpolation_operator}, then we have the strengthened supercloseness estimate $$\norm{u^G_h - u_h}_1 \leq Ch^{\frac{3p-2}{2p}}\norm{u}_{1,p}.$$

 \begin{proof}
    The first result essentially follows from the first Strang lemma combined with Lemma \ref{lemma:H1consistency} and \eqref{eq:assumptions_on_discrete_transport_operator}: \begin{align*}
         c_0\norm{u_h^G - u_h}_1 &\leq \sup_{v_h\in V_h} \frac{a_h(u_h - u_h^G,v_h) + r_h(u_h - u_h^G, v_h)}{\norm{v_h}_1} 
        \\&= \sup_{v_h\in V_h} \frac{\left<f,v_h\right> - a_h(u_h^G,v_h) - r_h(u_h^G,v_h)}{\norm{v_h}_1} 
        \\&= \sup_{v_h\in V_h} \frac{a(u_h^G,v_h) + r(u_h^G,v_h) - a_h(u_h^G,v_h) - r_h(u_h^G,v_h)}{\norm{v_h}_1} 
        \\&\leq Ch\norm{u_h^G}_1.
     \end{align*} 
Since standard estimates imply $\norm{u_h^G - u}_1 \leq C\inf_{v_h\in V_h} \norm{v_h - u}_1$, the first result follows. 

To prove the strengthened estimates, we simply modify the last line of the above computation by applying Theorem \ref{thm:superconvergence_lemma} and \eqref{eq:assumptions_on_discrete_transport_operator}: \begin{align*}
         c_0\norm{u_h^G - u_h}_1 &\leq  \sup_{v_h\in V_h} \frac{a(u_h^G,v_h) + r(u_h^G,v_h) - a_h(u_h^G,v_h) - r_h(u_h^G,v_h)}{\norm{v_h}_1} 
        \\&\leq Ch^{\frac{3p - 2}{2p}}\norm{u_h^G}_{1,p} \leq Ch^{\frac{3p - 2}{2p}}\norm{u}_{1,p},
\end{align*}
where the last inequality again follows from the standard estimate $\norm{u_h^G}_{1,p} \leq C\norm{u}_{1,p}$ for the Galerkin solution.

 
 \end{proof}
 \end{theorem}

 \begin{remark}
     Clearly, if the Galerkin solution is superconvergent --- that is, there is a constant $C(u)$ such that $\norm{u^G_h - I_h u}_{1} \leq C(u)h^{1 + s}$ for some $s > 0$ --- then Theorem \ref{thm:H1_reaction_diffusion_convergence} implies that $u_h$ is also superconvergent up to a maximum order of $h^{\frac{3p - 2}{2p}}$ when $u\in W^{1,p}(\Omega)$. If the reaction coefficient is smooth instead of discontinuous, then we obtain the optimal result $\norm{u_h^G - u_h}_1 \leq Ch^2\norm{u}_1$. Therefore, Theorem  \ref{thm:H1_reaction_diffusion_convergence} reveals that we expect to lose at least half an order of accuracy in the discontinuous case if our result is sharp. We provide numerical evidence to support the theoretical $O(h^{\frac{3}{2}})$ rate in Section \ref{section:numerics}.
 \end{remark}

 The next theorem, which gives an $L^2$-estimate, further requires that the solution $u\in H^2(\Omega)$, and that the adjoint problem $$a(v,w_\chi) = \left<\chi,v\right> \qquad \text{for all } v\in H^1_D(\Omega)$$ is $L^2$-regular in the sense that if $\chi\in L^2(\Omega)$, then $w_\chi\in H^2(\Omega)$ and there exists a constant $C > 0$ such that $\norm{w_\chi}_2 \leq C|\chi|_0$. We will also assume the existence of an interpolation operator $I_h$ satisfying \eqref{eq:interpolation_operator}, which may require, e.g., the nondegeneracy condition \eqref{eq:nondegeneracycondition}.

\begin{theorem}\label{thm:abstract_L2_convergence_theorem}
Assume the hypotheses stated above, i.e., that $u\in H^2(\Omega)$ and that the adjoint problem is $L^2$-regular. Then, there exists $C > 0$ such that $$|u - u_h|_0 \leq Ch^2\norm{u}_2.$$



 \begin{proof}
     Let $w_\chi$ be the solution of the adjoint problem for $\chi$. Then we have \begin{align*}
     \left<u_h - u , \chi\right> &= a(u_h - u , w_\chi) \\&= a(u_h - u, w_\chi - I_hw_\chi) + a(u_h - u ,I_hw_\chi)
     \\&= E_1 + a(u_h - u,I_hw_\chi) - a_h(u_h,I_hw_\chi) + a(u,I_hw_\chi) 
\\&= E_1 + a(u_h ,I_hw_\chi) - a_h(u_h,I_hw_\chi) 
\\&\leq \norm{u - u_h }_1\norm{w - I_hw_\chi}_1 + Ch^2\norm{u}_{2}\norm{w}_2 \\&\leq Ch^2\norm{u}_2\norm{w}_2.
\\&\leq Ch^2 \norm{u}_2|\chi|_0.
     \end{align*}
\end{proof}

 \end{theorem}



 \subsection{Tensor products and a fully monotone scheme}\label{section:fully_monotone}

After their introduction at the beginning of Section \ref{section:basiconvergenceresults}, the fact that the discretized domain is a tensor product of simplical meshes instead of a single simplical mesh has mostly faded into the background, because the essential properties (e.g., \eqref{eq:voxel-identities}) satisfied by the dual mesh on the tensor product mesh are inherited from its components. In the same way, our goal in this section will be to formally show how to produce a quadrature rule for the drift-diffusion operator from quadrature rules defined on each tensor product component, such that the tensor product discretization is monotone (or more narrowly, produces an $M$-matrix) whenever the component discretizations are. Our prime example is the edge-average finite element (EAFE) method. We will keep our explanations brief and narrowly focused on describing the extension of EAFE and other monotone schemes to tensor product meshes. We refer the reader to \cite{xu_monotone_1999} and the references therein for details on EAFE and the concept of monotonicity.


A monotone scheme for discretizing the reaction-drift-diffusion operator produces a stiffness matrix $A$ such that $A^{-1}\mathbf{f}$ is a vector with nonnegative entries whenever $\mathbf{f}$ has nonnegative entries. A sufficient condition for monotonicity is that $A$ is an $M$-matrix. The finite element Laplacian matrix for $P^1$ simplical elements is an $M$-matrix, and therefore monotone, whenever its off-diagonals are nonnegative, leading to the mesh condition described in \cite{xu_monotone_1999}. If the finite element Laplacian is an $M$-matrix, then the EAFE approximation of the drift-diffusion operator is also an $M$-matrix, and it remains so when perturbed by a nonnegative diagonal approximation of the reaction operator. In \cite{xu_monotone_1999}, the nonnegative diagonalization of the reaction operator is done using mass lumping, whereas here we have introduced our new quadrature rule \eqref{eq:quadrature_scheme} which can be applied for piecewise smooth functions.

\begin{remark}
    The mesh condition mentioned above is more restrictive in 3D than in 2D, since for the latter it suffices for the triangulation to be Delaunay. A subject for future research will be to investigate combining meshing techniques (e.g., using cut cells and related ideas) with modifications to EAFE to more easily produce an $M$-matrix even in the 3D case.
\end{remark}

We now extend the monotonicity property of the drift-diffusion operator to tensor products. To illustrate the basic idea, we return to the notation of Section \ref{section:basiconvergenceresults} and assume that the diffusion coefficient has the diagonal form $$\alpha(x) = \begin{pmatrix}
    \alpha_1(x)I_{d_1} &             &        & \\
                & \alpha_2(x)I_{d_2} &        & \\
                &             & \ddots & \\
                &             &        & \alpha_J(x) I_{d_J}
\end{pmatrix}\qquad \bbeta = \left[\bbeta_1(x),\bbeta_2(x),...,\bbeta_J(x)\right],$$ where $I_{d_j}$ is the $d_j\times d_j$ identity matrix and $\bbeta_j(x)\in \RR^{d_j}$. For PBSRD systems, each $\alpha_j$ would represent the diffusivity for a particular particle. The Galerkin finite element discretization of the drift-diffusion operator reads: $$A = \sum_{j=1}^J M^1 \otimes M^2 \otimes ... \otimes M^{j-1} \otimes A^{j} \otimes M^{j+1} \otimes ... \otimes M^{J},$$ where $\otimes$ represents a Kronecker product, $A^{j}$ discretizes the drift-diffusion operator $\nabla_{x^{(j)}} \cdot [\alpha_j(x) \nabla_{x^{(j)}}u + u\bbeta_j(x)]$ on $\mathcal{T}_h^{(j)}$, and $M^j$  is the $P^1$ finite element mass matrix on $\mathcal{T}_h^{(j)}$. 

To produce a monotone scheme, we simply replace each $M^i$ with its lumped counterpart, $\tilde{M}^i$, and replace the drift-diffusion operators $A^j$ with a quadrature approximation producing inverse-positive approximations $\tilde{A}^j$ of $A^j$. In terms of quadrature rules, the procedure described above involves expanding the drift-diffusion operator into a sum over component operators, and then applying mass lumping on each variable except those whose derivatives appear in that term. Below we give an example of such a derivation.

\begin{example}
When $J=2$, for example, we start with a two term expansion and then mass-lump each term in the variables which are not being differentiated: \begin{align*}
    \int_{\Omega_1 \times \Omega_2} &[\alpha(x,y) \nabla u_h(x,y) + u_h(x,y)\bbeta(x,y)] \cdot \nabla v_h(x,y)dxdy 
    \\&= \int_{\Omega_1 \times \Omega_2} [\alpha_1(x,y)\nabla_x u_h(x,y) + u_h(x,y)\bbeta_1(x,y)]\cdot \nabla_xv_h(x,y)dxdy\nonumber \\&+ \int_{\Omega_1 \times \Omega_2} [\alpha_2(x,y)\nabla_y u_h(x,y) + u_h(x,y)\bbeta_2(x,y)]\cdot \nabla_y(x,y)dxdy
    \\&\approx \sum_{i=1}^{N_2} |V_i^{(2)}|\int_{\Omega_1} [\alpha_1(x,x_i^{(2)})\nabla_x u_h(x,x_i^{(2)}) + u_h(x,x_i^{(2)})\bbeta_1(x,x_i^{(2)})]\cdot \nabla_x v_h(x,x^{(2)}_i)dx \\&+ \int_{\Omega_2} \sum_{i=1}^{N_1} |V_i^{(1)}|[\alpha_2(x_i^{(1)},y)\nabla_y u_h(x_i^{(1)},y) + u_h(x_i^{(1)},y)\bbeta_2(x_i^{(1)},y)]\cdot \nabla_y v_h(x_i^{(1)},y)dy
\end{align*}
Then, we apply mesh-dependent monotone discretizations $a^{(1)}_h$ and $a^{(2)}_h$ to the resulting component drift-diffusion operators: 

\begin{align*}
&\sum_{i=1}^{N_2} |V_i^{(2)}|\int_{\Omega_1} [\alpha_1(x,x_i^{(2)})\nabla_x u_h(x,x_i^{(2)}) + u_h(x,x_i^{(2)})\bbeta_1(x,x_i^{(2)})]\cdot \nabla_x v_h(x,x^{(2)}_i)dx \\&+ \int_{\Omega_2} \sum_{i=1}^{N_1} |V_i^{(1)}|[\alpha_2(x_i^{(1)},y)\nabla_y u_h(x_i^{(1)},y) + u_h(x_i^{(1)},y)\bbeta_2(x_i^{(1)},y)]\cdot \nabla_y v_h(x_i^{(1)},y)dy
\\&\approx \sum_{i=1}^{N_2} |V_i^{(2)}|a_h^{(1)}\left(u_h(\cdot ,x_{i}^{(2)}),v_h(\cdot, x_i^{(2)})\right) + \sum_{i=1}^{N_1} |V_i^{(1)}|a_h^{(2)}\left(u_h(x_{(i)}^{(1)}, \cdot),v_h(x_{i}^{(1)},\cdot)\right).
\end{align*}
Note that a similar derivation to the one here appears in \cite{heldman_convergent_2024} in the context of PBSRD simulation. 
\end{example}

Our discretization for the drift-diffusion operator is easily shown to be monotone, since it produces a block-diagonal matrix with monotone blocks. The stiffness matrix is also an $M$-matrix if the blocks are $M$-matrices. One can also easily show that if the block discretizations $a_h^{(i)}$ satisfy appropriate consistency error estimates \eqref{eq:assumptions_on_discrete_transport_operator}, then so do the bilinear forms associated with the tensor-product problem.

For the EAFE method that we use for our numerical examples in Section \ref{section:numerics}, only the first estimate from \eqref{eq:assumptions_on_discrete_transport_operator} has been proven in arbitrary dimensions \cite{xu_monotone_1999}. Other works (e.g., \cite{miller_scharfetter-gummel_1991}) have given numerical evidence that the EAFE method is second order accurate in the $L^2$-norm in two dimensions when the advection coefficient $\bbeta$ is a potential field, i.e., $\bbeta = \nabla \psi$ for some $\psi \in H^1(\Omega)$. We also mention here that in our own numerical experiments, when $\bbeta = \nabla \psi$ the local errors (quadrature errors on each element) are $O(h)$, which is the result shown in \cite{markowich_inverse-average-type_1988,xu_monotone_1999}. The global errors, however, are empirically $O(h^2)$. Finally, in one dimension one can actually prove stronger convergence results for EAFE along the lines of \eqref{eq:assumptions_on_discrete_transport_operator} \cite{ilin_differencing_1969,lazarov_finite_1996}. This implies that EAFE is second-order accurate when applied dimension-by-dimension on a structured grid as described above, which extends the results of, e.g., \cite{ilin_differencing_1969,lazarov_finite_1996} for one- and two-dimensional problems on structured grids. We leave it to future work to show that, in the general case, EAFE satisfies all of the estimates \eqref{eq:assumptions_on_discrete_transport_operator}.

\section{Implementation and numerics}\label{section:numerics}

For our numerical examples, we consider a drift-diffusion-reaction equation \eqref{eq:reaction-diffusion-parabolic-strong} with the following form: \begin{equation}\label{eq:numerical_example}
\begin{aligned}
    -\nabla \cdot [\nabla u(x) + u(x)\nabla \psi(x)] + \overline{\lambda} e^{-\psi(x)} u(x) 1_{K}(x) &= f(x) \\
    [\nabla u(x) + u(x)\nabla \psi(x)]  \cdot \bn_\Omega(x) &= 0.
\end{aligned}
\end{equation}
In the context of PBSRD models, \eqref{eq:numerical_example} is a simplified, forced version of the steady-state equations for a two-particle system with a single reversible reaction. See \cite{heldman_convergent_2024} for details.

In the notation used earlier in the paper, we have $\lambda(x) = \overline{\lambda} e^{-\psi(x)} $ and $\bbeta(x) = \nabla \psi(x)$. We also specify $\Omega = B_{1}(0)$ to be the unit ball in $\RR^d$, $K = B_{r^*}(0)$, $f(x) = 1$, and $\psi(x) = \overline{\kappa}|x|^2.$ We chose these parameters so that \eqref{eq:numerical_example} is radially symmetric, and therefore the solution $u(x) = u_d(x)$ to \eqref{eq:numerical_example} also solves the ODE \begin{equation}\label{eq:ode_sol}
\begin{aligned}
    -\frac{1}{r^{d-1}}\frac{d}{dr} r^{d-1} [u'(r) + u(r) \psi'(r)] + \overline{\lambda} e^{-\psi(r)} u(r) 1_{r < r^*} &= f(r) \\
    r'(0) = r'(1) &= 0.   
\end{aligned}
\end{equation}
We consider two test cases: a comparison between two different splittings of $\lambda$ with the exact solution computed from \eqref{eq:ode_sol}, and a comparison with the Galerkin solution in the $L^2$ and $H^1$ norms to test the superconvergence rate of $O(h^{\frac{3}{2}})$ predicted by Theorem \ref{thm:superconvergence_lemma}.

\subsection{Comparison with exact solution}

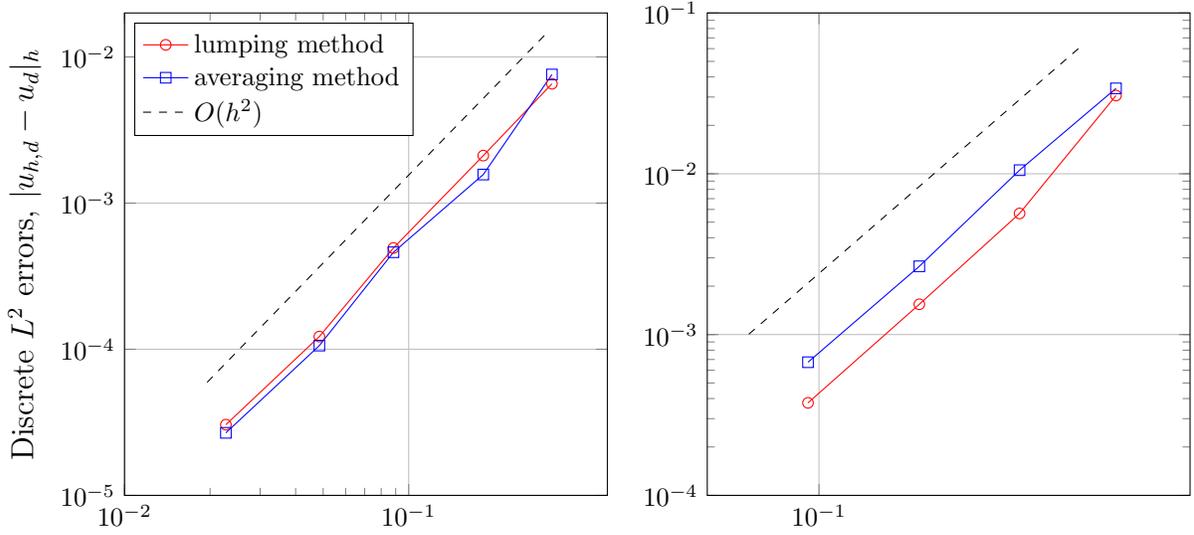
\begin{figure}[ht]
    \centering
\begin{tikzpicture}\input{sections2/conv-plots/convergenceplot1-2d}\end{tikzpicture}\hspace{10pt}\begin{tikzpicture}\input{sections2/conv-plots/convergenceplot1-3d}\end{tikzpicture}
    \caption{Convergence rate of the reaction-drift-diffusion finite element solution computed using the lumping method (red line, circular) markers, and the averaging method (blue line, square markers), compared with an $O(h^2)$ rate of convergence, with $\overline{\kappa} = 1$. The parameter $\overline{\kappa}$, represents the strength of the potential force and also influences the magnitude of the gradient for the smooth part of the reaction kernel.}\label{fig:convergence-plots-1}
    \end{figure}

\begin{figure}[ht]
    \centering
\begin{tikzpicture}\input{sections2/conv-plots/convergenceplot2-2d}\end{tikzpicture}\hspace{10pt}\begin{tikzpicture}\input{sections2/conv-plots/convergenceplot2-3d}\end{tikzpicture}
    \caption{Convergence rate of the reaction-drift-diffusion finite element solution computed using the lumping method (red line, circular) markers, and the averaging method (blue line, square markers), compared with an $O(h^2)$ rate of convergence, with $\overline{\kappa} = 5$. The parameter $\overline{\kappa}$, represents the strength of the potential force and also influences the magnitude of the gradient for the smooth part of the reaction kernel.}\label{fig:convergence-plots-2}
    \end{figure}
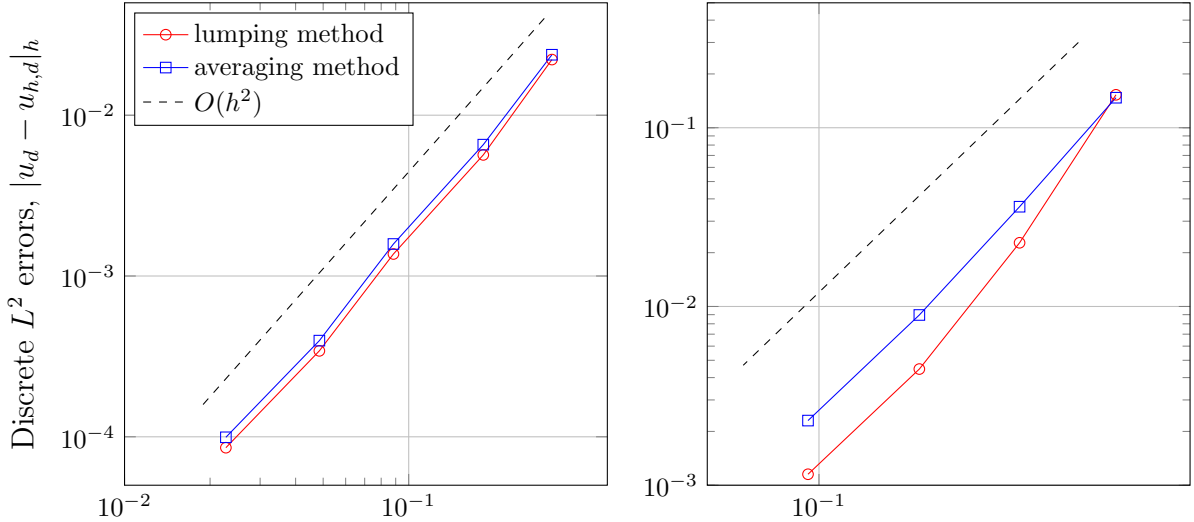

To compute the results in Figure \ref{fig:convergence-plots-1} and \ref{fig:convergence-plots-2}, we computed a highly accurate approximation to the solution $u_d$ of \eqref{eq:ode_sol} with $d=2,3$ using a finite element discretization with 10,000 elements on $[0,1]$. Then, we compared the result with the finite element solution $u_{h,d}$ of \eqref{eq:numerical_example} for different values of $h$. We used a piecewise linear approximation to the disk and spherical domain boundaries, which introduces an additional $O(h^2)$ approximation error to $u_{h,d}$. To test the effect of the splitting $\lambda(x) = \lambda_1(x)\lambda_2(x)$ as discussed in Remark \ref{remark:suboptimalerrorestimates}, for each mesh we also compare the solution computed using two discretization strategies: first, we take $\lambda_1(x) = \lambda(x)$, $\lambda_2 \equiv 1$, which we will call the (voxel-)averaging method, and second, we set $\lambda_2(x) = \lambda(x)$, $\lambda_1\equiv 1$, which we call the (mass-)lumping method. In any case, the finite element discretization of \eqref{eq:numerical_example} leads to a matrix system of the form $(A + R)\mathbf{u}_{h,d} = \mathbf{f}_h,$ where $\mathbf{u}_{h,d}$ is the vector of coefficients for $u_{h,d}$, $A$ is the discrete transport operator, $R$ is the discrete reaction operator, and $\mathbf{f}_h$ is the right-hand side. The vectors and matrices have entries 
\begin{align}
    A_{ji} &= \begin{cases}
        \frac{\psi_j - \psi_i}{e^{\psi_j - \psi_i}}\int_{\Omega} \nabla \phi_i(x) \cdot \nabla \phi_j(x) dx & i\neq j,\\
        \sum_{k\neq i} -A_{ki} & i = j,\label{eq:EAFE}
    \end{cases} \\
    R_{ij} &= \begin{cases} 0 & i\neq j \\
                           \overline{\lambda}\int_{V_i \cap K} e^{-\psi(x)}dx & \text{Using the averaging method}, \\
                           \overline{\lambda}e^{-\psi_i}|V_i \cap K| & \text{Using the lumping method},\nonumber
            \end{cases}\\
    (\mathbf{f}_h)_i &= \int_{\Omega} f(x)\phi_i(x)dx = |V_i|.\nonumber
\end{align}
The entries of $A$ shown above were computed using a version of the edge-average finite element method \cite{xu_monotone_1999,markowich_inverse-average-type_1988}, mentioned earlier in Section \ref{section:reactiondiffusionequations}. The entries of $R$ are computed using the averaging or lumping method as follows: for the averaging method, the integrals $\overline{\lambda}\int_{V_i} e^{-\psi(x)}1_K(x) dx$ are computed using a Monte Carlo method on $V_i$ with $\lfloor\frac{1000}{h^2}\rfloor$ sample points. The algorithms we used to sample from the voxels $V_i$ for the Monte Carlo integration and to compute the intersection areas $|V_i \cap K|$ are implemented in a code developed for CRDME simulations available at \href{https://gitlab.com/mheldman/crdme}{https://gitlab.com/mheldman/crdme}. The sampling algorithm will be presented in a forthcoming paper \cite{heldman_rejection_2024}, while algorithms for computing the intersection between $B_{r^*}(0)$ and $V_i$ are described in \cite{isaacson_unstructured_2018} (2D version) and \cite{strobl_exact_2016,heldman_convergent_2024} (3D version). For other domains $K$, we expect that an analysis along the lines of, e.g., \cite{kashiwabara_penalty_2016,kashiwabara_finite_2023} can justify the approximation of $K$ by a piecewise linear function; then, the intersections between polyhedra/polygons can be computed using a discrete geometry library such as CGAL \cite{the_cgal_project_cgal_2023}. 
\begin{remark}
    We emphasize that we are not suggesting using Monte Carlo to compute the integrals is a good strategy. However, we are not aware of other good algorithms for computing general integrals like $\int_{V_i\cap K} \lambda_2(x)dx$ other than $h$-refinement, which already can be less efficient than Monte Carlo in three dimensions due to the curse of dimensionality. As we outline for $K = B_{r^*}(0)$ above, in the case $\lambda_2(x) \equiv 1$ the exact evaluation of each $|V_i \cap K|$ can be $O(1)$ (though still expensive), and so far this seems to be the most feasible use case for the method.
\end{remark}

 In Figure \ref{fig:convergence-plots-1} and \ref{fig:convergence-plots-2}, we observe $O(h^2)$ convergence of the discrete solution in both the 2D and 3D simulations, as measured in the relative discrete $L^2(\Omega)$-norm between $u_{h,d}$ and the true solution $u_{d}$: $$\frac{|u_d - u_{h,d}|_h}{|u_d|_h} = \sqrt{\frac{\mathcal{M}_h(|u_d - u_{h,d}|^2)}{\mathcal{M}_h(|u_d|^2)}} = \sqrt{\frac{\sum_{i=1}^N |V_i||u_d(x_i) - u_{h,d}(x_i)|^2}{\sum_{i=1}^N |V_i||u_d(x_i)|^2}}.$$ 
The solutions in Figure \ref{fig:convergence-plots-1} were computed with parameters $\overline{\lambda} = 5$, $\overline{\kappa} = 1$ and $r^* = \frac{\pi}{5}$, while in Figure \ref{fig:convergence-plots-2} we kept $\overline{\lambda}$ and $r^*$ the same and increased the force to $\overline{\kappa} = 5$ to test the effect of an increased gradient for $\lambda(x)$ on the errors for the averaging and lumping methods. Although for both methods the error increases with $\overline{\kappa} = 5$, the overall trend remains largely the same. Based on the data pictured below, there seems to be little advantage to be gained from using the averaging method instead of the lumping method despite the fact that the averaging method produces more optimistic error bounds (in terms of $W^{1,\infty}(\Omega)$-norm of $\lambda$ instead of the $W^{2,\infty}(\Omega)$ norm, see Remark \ref{remark:suboptimalerrorestimates}). On the other hand, the averaging method is well-defined for a wider class of functions. We leave a more detailed analysis comparing the two methods to future work.

\subsection{Supercloseness: comparison with Galerkin solution}

For a second test case, we simplify further to the pure reaction-diffusion problem ($\psi \equiv 0$), keeping $\bar{\kappa} = 5$. The discrete solution $u_{h,d}$ is computed using the same formulae and methodology described in the previous subsection, where we note that the averaging and lumping method coincide for $\psi \equiv 0$. We compare with the Galerkin solution $u^G_{h,d}$.

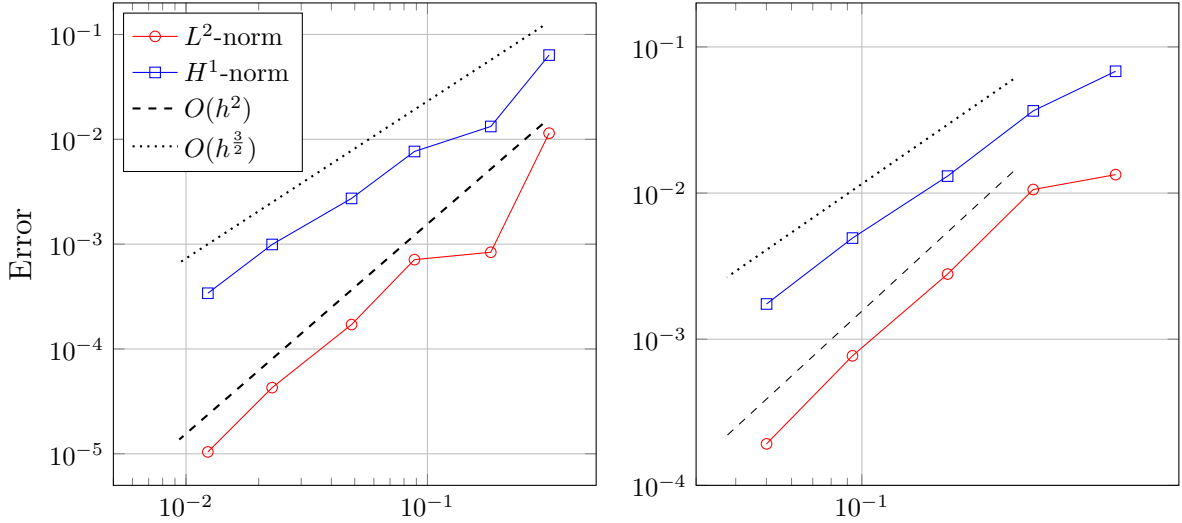
\begin{figure}[ht]
    \centering
\begin{tikzpicture}\input{sections2/conv-plots/convergenceplot1-2d-H1}\end{tikzpicture}\hspace{10pt}\begin{tikzpicture}\input{sections2/conv-plots/convergenceplot1-3d-H1}\end{tikzpicture}
    \caption{$L^2$-norms of $u_h - u_h^G$ (red line, circular markers) and $\nabla(u_h - u_h^G)$ (blue line, square markers) for the 2D (left) and 3D (right) problem, where $u_h^G$ is the Galerkin solution of the drift-diffusion equation and $u_h$ is computed using the finite-volume quadrature rule. The $H^1$ errors are $O(h^{\frac{3}{2}})$, a half-order reduction in the convergence rate compared to when the reaction coefficient is smooth.}\label{fig:convergence-plots-3}
    \end{figure}

To compute the Galerkin reaction operator $R_{ij} = r(\phi_i,\phi_j)$ requires us to evaluate $\int_{\Omega} \phi_i\phi_j 1_Kdx$ for each pair of basis functions $\phi_i,\phi_j$. We do so in the usual element-by-element way, incorporating the discontinuity by dividing the elements into three cases: 
\begin{enumerate}
    \item If each vertex $x_i$ of $T$ is contained in $B_{r^*}(0)$, then $T \subseteq B_{r^*}(0)$. Therefore, we use a quadrature formula to exactly compute $\int_{T} \phi_i\phi_j dx$ for each vertex pair $x_i,x_j\in T$.
    \item If each vertex $x_i$ of $T$ is in $B^c_{\bar{r}}(0)$, where $\bar{r} = \sqrt{(h/2)^2 + (r^*)^2}$ (i.e., $|x_i| > \bar{r}$), then $T\cap B_{r^*}(0) = \emptyset$. Therefore, we set $\int_{T} \phi_i\phi_j 1_Kdx = 0$ for all $x_i,x_j\in T$.
    \item Otherwise, we compute $\int_{T} \phi_i\phi_j 1_Kdx$ by Monte Carlo integration, again using $\lfloor\frac{1000}{h^2}\rfloor$ sample points.
\end{enumerate}
In Figure \ref{fig:convergence-plots-3}, we observe that the $L^2(\Omega)$ errors (again computed in the discrete norm $|\cdot |_h$) between $u_{h,d}$ and $u_{h,d}^G$ match to the optimal $O(h^2)$ accuracy, while in the $H^1(\Omega)$ norm the convergence order is reduced to $O(h^{\frac{3}{2}})$. Both of these rates match the theoretical predictions of Theorems \ref{thm:abstract_L2_convergence_theorem} and \ref{thm:superconvergence_lemma}; in particular, they provide evidence that the final estimate of Corollary \ref{corollary:final_global_error_estimate} is sharp in that, for example, the $H^2(\Omega)$-norm on $v$ cannot be replaced by an $H^1(\Omega)$-norm.

\section{Conclusion}\label{section:conclusion}

We have proven convergence estimates for a quadrature scheme for discontinuous integrands, which we applied to construct a monotone finite-element discretization of scalar reaction-drift-diffusion equations with discontinuous coefficients. The quadrature scheme was originally intended and successfully used for discretizing systems of reaction-drift-diffusion equations related to PBSRD simulation \cite{isaacson_convergent_2013,isaacson_unstructured_2018,isaacson_unstructured_2024,heldman_convergent_2024}, and our work justifies theoretically the empirical $O(h^2)$ convergence estimates observed in those works. More generally, our work shows that incorporating averaging-type quadrature schemes into low-order finite-element methods for PDEs with irregular coefficients can recover additional orders of convergence beyond what can be recovered using a traditional element-by-element analysis of the quadrature errors.

Several possible modifications and extensions of our work, some of which are mentioned earlier in the paper, include:

\begin{itemize}
    \item Rigorous comparisons of the averaging and lumping schemes.
    \item Exploration of alternative definitions of the dual mesh such as \eqref{eq:nodal-support}.
    \item Proof of second-order accuracy for the EAFE discretization.
    \item Incorporation of our results into CRDME convergence proofs (i.e., convergence of the discretization described in \cite{heldman_convergent_2024}).
    \item Extension of our results to include irregularities in other coefficients, such as the drift coefficient $\bbeta$.
\end{itemize}

\appendix

\section*{Appendix} \label{section:appendix}

\begin{proof}[Proof of Lemma \ref{lemma:characterizationofminimizers}]
    Up to a rotation of the domain, there exists a Lipschitz continuous parametrization $(y,\varphi(y))$, $y\in V\subseteq \mathbb{R}^{d-1}$ of $\partial K$ around $P_K(x) $ such that $V\cap \Omega = \{(y,y_n) : y_n < \varphi(y)\} $. We have that $P_K(x) = (y^*,\varphi(y^*))$, where $$y^* = \argmin_{y\in V} \Phi(y) :=  \frac{1}{2}\norm{\left<y, \varphi(y)\right> - x}^2.$$ Since $\Phi$ is Lipschitz continuous and $y^*$ is a critical point, we have that $0 \in \partial \Phi(y^*)$, where \begin{align*}\partial \Phi(y^*) = \conv \left\{ \lim_{i\to\infty} \nabla \Phi(y_i) :  \text{$(y_i)_{i=1}^\infty\subseteq W$, $y_i\to y^*$, $\nabla \Phi(y_i)$ converges}\right\}\end{align*} is the Clarke subdifferential of $\Phi$ at $y^*$ with $W$ denoting the set of points in $V$ where $\Phi$ is differentiable  (see \cite{clarke_optimization_1990}, p. 27 for the definition of the Clarke subdifferential, Theorem 2.5.1 for the relevant characterization, and Proposition 2.3.2 for the relationship to local extrema). Now, at any $y$ such that $\varphi$ is differentiable, the gradient of $\Phi$ is given by $$\nabla\Phi(y) = y - \tilde{x} + \nabla\varphi(y)(\varphi(y) - x_n),$$ where we use $\tilde{x}$ to denote the projection of $x$ into $\mathbb{R}^{n-1}$. Using this formula for the gradient and the characterization of $\partial \Phi(y^*)$ given above, the inclusion $0\in \partial \Phi(y^*)$ can be rewritten as $$\tilde{x} - y^* = \sum_{j=1}^m \beta_j (\lim_{i\to \infty} -\nabla \varphi(y_i))(x_n - \varphi(y^*))$$ for some scalars $\beta_j \geq 0$ with $\sum_{j=1}^m \beta_j = 1$. Using the fact that $n_K(y_i,\varphi(y_i)) = \frac{\left<-\nabla \varphi(y_i),1\right>}{\sqrt{\norm{\nabla \varphi(y_i)}^2 + 1}}$, the above is evidently equivalent to $$x - P_K(x) = (x_n - \varphi(y^*))\sum_{j=1}^m\beta_j\lim_{i\to\infty} \left<-\nabla \varphi(y_i),1\right> = \sum_{j=1}^m\beta_jc_j\lim_{i\to\infty} n_K(y_i,\varphi(y_i)),$$ where $c_j = \lim_{i\to\infty} \sqrt{\norm{\nabla \varphi(y_i)}^2 + 1}$. That the righthand side is nonzero follows from the fact that $x - P_K(x) \neq 0$. Therefore, setting $\alpha_j = \beta_j c_j$ and normalizing both sides yields the result.
\end{proof}

%% file: sections2/conv-plots/convergenceplot1-2d.tex
     \pgfplotsset{
    convaxis/.style={
      compat=newest,
      width=8cm,
      height=8cm,
      xmin=1e-2,
      xmax=5e-1,
      ymin=1e-5,
      ymax=2e-2,
      xtick={1e-3,1e-2,1e-1},
      ytick={1,1e-1, 1e-2,1e-3,1e-4,1e-5},
      grid=major,
      tick label style={font=\normalsize},
      title style={align=center, font=\footnotesize},
      xlabel style={align=center, font=\normalsize},
      ylabel style={align=center, font=\normalsize},
      legend style={at={(0.02,0.98)}, anchor=north west, font=\normalsize},
      legend cell align=left,
      legend entries={{$\text{lumping method}$}, {$\text{averaging method}$}, {$O(h^2)$}}
    },
    convaxisscale/.style={
      /pgfplots/xmode=log,
      /pgfplots/ymode=log,
    },
    convaxisscale/.belongs to family=/pgfplots/scale
  }
      \def\colorA{red}
      \def\colorB{black}
      \def\colorC{blue}
      \def\colorD{orange}
      \def\scaleA{.1}
      \def\scaleB{2}
      \def\data{sections2/data/2d/table-1.txt}
    \begin{axis}[
    convaxis, convaxisscale,
    ylabel={\Large Discrete $L^2$ errors, $|u_{h,d} - u_d|_h$},
  ]
    \addplot[color=\colorA, mark=o] table [x=h,y=lump] {\data};  
    \addplot[color=\colorC, mark=square] table [x=h,y=integrate] {\data};
     \addplot[dashed, color=\colorB] coordinates {
       (.3, 2*.007)
       (.3/2^4, 2*.007/4^4)
       }; 

      convaxisscale/.style={
      /pgfplots/xmode=log,
      /pgfplots/ymode=log,
    };
    \end{axis}

%% file: sections2/conv-plots/convergenceplot1-3d.tex
     \pgfplotsset{
    convaxis/.style={
      compat=newest,
      width=8cm,
      height=8cm,
      xmin=5e-2,
      xmax=1,
      ymin=1e-4,
      ymax=1e-1,
      xtick={1e-2,1e-1,2},
      ytick={},
      grid=major,
      tick label style={font=\normalsize},
      title style={align=center, font=\footnotesize},
      xlabel style={align=center, font=\normalsize},
      ylabel style={align=center, font=\normalsize},
    },
    convaxisscale/.style={
      /pgfplots/xmode=log,
      /pgfplots/ymode=log,
    },
    convaxisscale/.belongs to family=/pgfplots/scale
  }
      \def\colorA{red}
      \def\colorB{black}
      \def\colorC{blue}
      \def\colorD{orange}
      \def\scaleA{.1}
      \def\scaleB{2}
      \def\data{sections2/data/3d/table-1.txt}
    \begin{axis}[
    convaxis, convaxisscale,
  ]
    \addplot[color=\colorA, mark=o] table [x=h,y=lump] {\data};  
    \addplot[color=\colorC, mark=square] table [x=h,y=integrate] {\data};
     \addplot[dashed, color=\colorB] coordinates {
       (.5, 2*.03)
       (.5/2^3, 2*.03/4^3)
       }; 

      convaxisscale/.style={
      /pgfplots/xmode=log,
      /pgfplots/ymode=log,
    };
    \end{axis}

%% file: sections2/conv-plots/convergenceplot2-2d.tex
     \pgfplotsset{
    convaxis/.style={
      compat=newest,
      width=8cm,
      height=8cm,
      xmin=1e-2,
      xmax=5e-1,
      ymin=5e-5,
      ymax=5e-2,
      xtick={1e-3,1e-2,1e-1},
      ytick={1,1e-1, 1e-2,1e-3,1e-4,1e-5},
      grid=major,
      tick label style={font=\normalsize},
      title style={align=center, font=\footnotesize},
      xlabel style={align=center, font=\normalsize},
      ylabel style={align=center, font=\normalsize},
      legend style={at={(0.02,0.98)}, anchor=north west, font=\normalsize},
      legend cell align=left,
      legend entries={{$\text{lumping method}$}, {$\text{averaging method}$}, {$O(h^2)$}}
    },
    convaxisscale/.style={
      /pgfplots/xmode=log,
      /pgfplots/ymode=log,
    },
    convaxisscale/.belongs to family=/pgfplots/scale
  }
      \def\colorA{red}
      \def\colorB{black}
      \def\colorC{blue}
      \def\colorD{orange}
      \def\scaleA{.1}
      \def\scaleB{2}
      \def\data{sections2/data/2d/table-2.txt}
    \begin{axis}[
    convaxis, convaxisscale,
    ylabel={\Large Discrete $L^2$ errors, $|u_d - u_{h,d}|_h$},
  ]
    \addplot[color=\colorA, mark=o] table [x=h,y=lump] {\data};  
    \addplot[color=\colorC, mark=square] table [x=h,y=integrate] {\data};
     \addplot[dashed, color=\colorB] coordinates {
       (.3, 2*.02)
       (.3/2^4, 2*.02/4^4)
       }; 

      convaxisscale/.style={
      /pgfplots/xmode=log,
      /pgfplots/ymode=log,
    };
    \end{axis}

%% file: sections2/conv-plots/convergenceplot2-3d.tex
     \pgfplotsset{
    convaxis/.style={
      compat=newest,
      width=8cm,
      height=8cm,
      xmin=5e-2,
      xmax=1,
      ymin=1e-3,
      ymax=5e-1,
      xtick={1e-2,1e-1,2},
      ytick={},
      grid=major,
      tick label style={font=\normalsize},
      title style={align=center, font=\footnotesize},
      xlabel style={align=center, font=\normalsize},
      ylabel style={align=center, font=\normalsize},
    },
    convaxisscale/.style={
      /pgfplots/xmode=log,
      /pgfplots/ymode=log,
    },
    convaxisscale/.belongs to family=/pgfplots/scale
  }
      \def\colorA{red}
      \def\colorB{black}
      \def\colorC{blue}
      \def\colorD{orange}
      \def\scaleA{.1}
      \def\scaleB{2}
      \def\data{sections2/data/3d/table-2.txt}
    \begin{axis}[
    convaxis, convaxisscale,
  ]
    \addplot[color=\colorA, mark=o] table [x=h,y=lump] {\data};  
    \addplot[color=\colorC, mark=square] table [x=h,y=integrate] {\data};
     \addplot[dashed, color=\colorB] coordinates {
       (.5, 2*.15)
       (.5/2^3, 2*.15/4^3)
       }; 

      convaxisscale/.style={
      /pgfplots/xmode=log,
      /pgfplots/ymode=log,
    };
    \end{axis}

%% file: sections2/conv-plots/convergenceplot1-2d-H1.tex
     \pgfplotsset{
    convaxis/.style={
      compat=newest,
      width=8cm,
      height=8cm,
      xmin=5e-3,
      xmax=5e-1,
      ymin=5e-6,
      ymax=2e-1,
      xtick={1e-3,1e-2,1e-1},
      ytick={1,1e-1, 1e-2,1e-3,1e-4,1e-5},
      grid=major,
      tick label style={font=\normalsize},
      title style={align=center, font=\footnotesize},
      xlabel style={align=center, font=\normalsize},
      ylabel style={align=center, font=\normalsize},
      legend style={at={(0.02,0.98)}, anchor=north west, font=\normalsize},
      legend cell align=left,
      legend entries={{$\text{$L^2$-norm}$}, {$\text{$H^1$-norm}$}, {$O(h^2)$}, {$O(h^{\frac{3}{2}})$}}
    },
    convaxisscale/.style={
      /pgfplots/xmode=log,
      /pgfplots/ymode=log,
    },
    convaxisscale/.belongs to family=/pgfplots/scale
  }
      \def\colorA{red}
      \def\colorB{black}
      \def\colorC{blue}
      \def\colorD{orange}
      \def\scaleA{.1}
      \def\scaleB{2}
      \def\data{sections2/data/2d/table-1-h1.txt}
    \begin{axis}[
    convaxis, convaxisscale,
    ylabel={\Large Error},
  ]
    \addplot[color=\colorA, mark=o] table [x=h,y=L2] {\data};  
    \addplot[color=\colorC, mark=square] table [x=h,y=H1] {\data};
     \addplot[dashed,thick, color=\colorB] coordinates {
       (.3, 2*.007)
       (.3/2^5, 2*.007/4^5)
       }; 

     \addplot[dotted,thick,color=\colorB] coordinates {
       (.3, 2*.06)
       (.3/2^5, 2*.06/2^7.5)
       }; 
      convaxisscale/.style={
      /pgfplots/xmode=log,
      /pgfplots/ymode=log,
    };
    \end{axis}

%% file: sections2/conv-plots/convergenceplot1-3d-H1.tex
     \pgfplotsset{
    convaxis/.style={
      compat=newest,
      width=8cm,
      height=8cm,
      xmin=3e-2,
      xmax=1,
      ymin=1e-4,
      ymax=2e-1,
      xtick={1e-3,1e-2,1e-1},
      ytick={1,1e-1, 1e-2,1e-3,1e-4,1e-5},
      grid=major,
      tick label style={font=\normalsize},
      title style={align=center, font=\footnotesize},
      xlabel style={align=center, font=\normalsize},
      ylabel style={align=center, font=\normalsize},
      legend style={at={(0.02,0.98)}, anchor=north west, font=\normalsize},
      legend cell align=left,
    },
    convaxisscale/.style={
      /pgfplots/xmode=log,
      /pgfplots/ymode=log,
    },
    convaxisscale/.belongs to family=/pgfplots/scale
  }
      \def\colorA{red}
      \def\colorB{black}
      \def\colorC{blue}
      \def\colorD{orange}
      \def\scaleA{.1}
      \def\scaleB{2}
      \def\data{sections2/data/3d/table-1-h1.txt}
    \begin{axis}[
    convaxis, convaxisscale,
  ]
    \addplot[color=\colorA, mark=o] table [x=h,y=L2] {\data};  
    \addplot[color=\colorC, mark=square] table [x=h,y=H1] {\data};
     \addplot[dashed, color=\colorB] coordinates {
       (.3, 2*.007)
       (.3/2^3, 2*.007/4^3)
       }; 

     \addplot[dotted,thick,color=\colorB] coordinates {
       (.3, 2*.03)
       (.3/2^3, 2*.03/2^4.5)
       }; 
      convaxisscale/.style={
      /pgfplots/xmode=log,
      /pgfplots/ymode=log,
    };
    \end{axis}